\documentclass[a4]{amsart}
\usepackage[usenames]{color}

\usepackage[normalem]{ulem}

\usepackage{amsmath,amssymb,amsthm,amsfonts,enumerate,url,slashbox,mathbbol,mathrsfs,tikz,graphicx,pifont}
\newcommand{\cmark}{\ding{51}}%
\newcommand{\xmark}{\ding{55}}%

\usepackage{pgfplots}
\usepgfplotslibrary{polar}

\usepackage[utf8]{inputenc}

\usepackage{verbatim}
\usepackage{array}

\theoremstyle{plain}
\newtheorem{theorem}{Theorem}[section]
\newtheorem{lemma}[theorem]{Lemma}
\newtheorem{proposition}[theorem]{Proposition}
\newtheorem{corollary}[theorem]{Corollary}

\theoremstyle{definition}
\newtheorem{definition}[theorem]{Definition}
\newtheorem{example}[theorem]{Example}

\newtheorem{assumption}[theorem]{Assumption}

\newtheorem{remark}[theorem]{Remark}

\numberwithin{equation}{section}
\numberwithin{figure}{section}

\DeclareMathOperator{\dist}{dist}

\DeclareMathOperator{\dista}{dist}

\DeclareMathOperator{\essinf}{ess \, inf}

  \def\mG{\mathsf{G}}
   \def\mF{\mathsf{F}}
  \def\mV{\mathsf{V}}
  \def\mE{\mathsf{E}}

  \def\mP{\mathsf{P}}
  \def\mK{\mathsf{K}}
  \def\mS{\mathsf{S}}
  \def\mI{\mathsf{I}}
  
  \def\mT{\mathsf{T}}
 \def\mv{\mathsf{v}}
 \def\me{\mathsf{e}}
 \def\mw{\mathsf{w}}

\newcommand{\R}{\mathbb{R}}
\newcommand{\N}{\mathbb{N}}

\def\XXint#1#2#3{{\setbox0=\hbox{$#1{#2#3}{\int}$ }
\vcenter{ \hbox{$#2#3$} }\kern-.45\wd0}}

\title{On the spectral gap of a quantum graph} 

\subjclass[2010]{}


\author[J.B.~Kennedy]{James B. Kennedy}
\author[P.~Kurasov]{Pavel Kurasov}
\author[G.~Malenová]{Gabriela Malenov\'a}
\author[D.~Mugnolo]{Delio Mugnolo}

\address{James B.\ Kennedy, Institut f\"ur  Analysis, Dynamik und Modellierung, Universit\"at Stuttgart, Pfaffenwaldring 57, D-70569 Stuttgart, Germany}
\email{james.kennedy@mathematik.uni-stuttgart.de}

\address{Pavel Kurasov, Department of Mathematics, Stockholm University, SE-106 91 Stockholm, Sweden}
\email{kurasov@math.su.se}

\address{Gabriela Malenová, Department of Mathematics, KTH Stockholm, SE-100 44 Stockholm, Sweden}
\email{malenova@kth.se}

\address{Delio Mugnolo, Lehrgebiet Analysis, Fakult\"at Mathematik und Informatik, Fern\-Universit\"at in Hagen, D-58084 Hagen, Germany}
\email{delio.mugnolo@fernuni-hagen.de}

\date{\today}

\thanks{Part of this work was completed while J.B.K.~was the recipient of a fellowship of the Alexander von Humboldt Foundation, Germany. G.M.\ and D.M.\ were partially supported by the Land Baden--W\"urttemberg in the framework of the \emph{Juniorprofessorenprogramm} -- research project on ``Symmetry methods in quantum graphs''. All four authors were partially supported by the Center for Interdisciplinary Research (ZiF) in Bielefeld in the framework of the cooperation group on ``Discrete and continuous models in the theory of networks". The authors would like to thank Jens Wirth for helpful discussions regarding Theorem~\ref{th:d-upper}.}

\begin{document}

\begin{abstract}
We consider the problem of finding universal bounds of ``isoperimetric'' or ``isodiametric'' type on the spectral gap of the Laplacian on a metric graph with natural boundary conditions at the vertices, in terms of various analytical and combinatorial properties of the graph: its total length, diameter, number of vertices and number of edges. We investigate which combinations of parameters are necessary to obtain non-trivial upper and lower bounds and obtain a number of sharp estimates in terms of these parameters. We also show that, in contrast to the Laplacian matrix on a combinatorial graph, no bound depending only on the diameter is possible. 
As a special case of our results on metric graphs, we deduce estimates for the normalised Laplacian matrix on combinatorial graphs which, surprisingly, are sometimes sharper than the ones obtained by purely combinatorial methods in the graph theoretical literature.
\end{abstract}

\maketitle

\section{Introduction}
\label{sec:intro}

A classical question in spectral theory consists in determining for which bodies --  among all those with prescribed volume, or surface measure, or perhaps another relevant geometric quantity -- a given combination of eigenvalues of the associated Laplacian, say with Dirichlet or Neumann boundary conditions, is maximised or minimised. In the case of domains, this goes back as far as 1870 to a now-famous conjecture of Lord Rayleigh, answered in the affirmative by G.\ Faber in 1923 in the planar case and E.\ Krahn in 1926 in the general case, that the first eigenvalue $\lambda_1$ of the Laplacian on $\Omega \subset \R^d$ with Dirichlet boundary conditions is always at least as large as that of a ball in $\R^d$ with the same volume, with equality being attained if and only if $\Omega$ is in fact a $d$-dimensional ball. This result was arguably the starting point of spectral geometry. Analogous results for the first positive eigenvalue $\lambda_1$ of the Laplacian with Neumann boundary conditions were proved by G.\ Szeg\H{o} and H.F.\ Weinberger in 1954 and 1956, respectively: in the Neumann case, it turns out that the first (non-trivial) eigenvalue is \emph{maximal} if $\Omega$ is a $d$-dimensional ball. We refer to~\cite{Hen06} for a survey of such \emph{isoperimetric inequalities} for differential operators.

The importance of the first positive eigenvalue cannot be overstated. For example, in the theory of parabolic equations $\lambda_1$ gives the speed of convergence of the system towards equilibrium. In mathematical physics it is the energy level associated with the ground state of the system, or the first excited state, if Neumann conditions are imposed; in the latter case $\lambda_1$ is therefore often referred to as the \emph{spectral gap}.

In recent years it has become increasingly clear that there exist parallel differential geometrical theories in the continuous setting of manifolds and in the discrete setting of graphs, see e.g.\ the survey~\cite{Kel15}. In the case of graphs, however, there are several competing notions which can be considered as generalisations of the Laplacian, including the \emph{discrete Laplacian} $\mathcal L$ and the \emph{normalised Laplacian} $\mathcal L_{\rm norm}$; moreover, it is not quite indisputable which geometric quantities should be chosen in order to impose meaningful restrictions on the class of sets under consideration. For instance, it was already proved by M.\ Fiedler in~\cite{Fie73} that among all connected graphs on a given number of vertices, the first positive eigenvalue of the discrete Laplacian $\mathcal L$ is maximal (resp., minimal) in the case of the complete (resp., path) graph. A more delicate analysis is needed to discuss the cases of graphs whose number of vertices \emph{and} edges, or else whose number of vertices of degree one,  is prescribed, see~\cite{LalPatSah11,BiyLey12}. Comparable results are known for the normalised Laplacian $\mathcal L_{\rm norm}$, cf.\ Section~\ref{sec:discrete}.

In the present paper we are going to focus on quantum graphs: roughly speaking, a (compact) \emph{quantum graph} is a (finite, connected) graph, each of whose edges $\me$ is identified with an interval of $\mathbb R$ of (finite) length $|\me|$. Then, the usual Euclidean distance on each edge induces in a natural way a metric space structure on a quantum graph -- we refer the reader to the monographs~\cite{BerKuc13,Mug14,Kur15} for more details. It is thus possible to define on each such interval a differential operator (which plays the role of a Hamiltonian in the framework of quantum mechanics on graphs): here we will only focus on the case of (one-dimensional) Laplacians. Gluing all these operators together by means of suitable boundary conditions yields a new Laplacian-type operator, the subject of our investigations.

It is mathematical folklore that quantum graph Laplacians ``interpolate'' between Laplace--Beltrami operators on compact manifolds and normalised Laplacians on combinatorial graphs; indeed, interesting relations have been proved using these interplays. It is perhaps surprising that very few results are known in the area of spectral geometry for quantum graphs: possibly the only Faber--Krahn-type result for quantum graphs says that the lowest non-trivial eigenvalue $\lambda_1$, i.e.\ the spectral gap, of the Neumann (i.e.\ Kirchhoff) Laplacian on a quantum graph is {\em minimised} among all graphs of given total length by the path, cf.~\cite{Fri05,KurNab14,Nic87}  (not maximised, as in the case of domains).

Our principal aim here is to undertake a more systematic investigation of universal eigenvalue inequalities for the quantum graph Laplacian with natural conditions at the vertices. Since we wish to gain a sense for which problems are (mathemtically) ``natural'' or ``sensible'', thereby also laying the foundations for future work, we will restrict ourselves to the prototype problem of upper and lower estimates on the spectral gap and to what we consider to be the four most natural quantities (see Section~\ref{sec:basic} below for precise definitions): 
\begin{itemize}
\item the total length $L$ of a graph, 
\item its diameter $D$, 
\item the number $V$ of its vertices and
\item the number $E$ of its edges.
\end{itemize}

In fact, even in the most elementary cases it turns out to be a surprisingly subtle question as to which problems are well posed; as we shall attempt to show below, quantum graphs can in fact display types of behaviour that are in a sense more complex than those of manifolds and combinatorial graphs. For up-to-date and fairly comprehensive overviews of the currently known estimates for the spectral gap of the discrete and normalised Laplacians of a combinatorial graph, we refer to~\cite{Mol12} and in~\cite{Chu97}, respectively.

If all edges of a quantum graph have the same length, i.e.\ if the graph is equilateral, then J.\ von Below showed in~\cite[Theorem, page 320]{Bel85} that all spectral problems concerning the Laplacian can be equivalently reduced to corresponding spectral problems for the normalised Laplacian $\mathcal L_{\rm norm}$ on the underlying combinatorial graph (see Section~\ref{sec:discrete} for a brief definition). In particular, the lowest non-zero eigenvalue $\lambda_1$ of the quantum graph Laplacian agrees with
\begin{equation}\label{eq:below}
\lambda_1 = \arccos^2 (1-\alpha_1)\qquad \hbox{provided }\alpha_1\in [0,2)\ ,
\end{equation}
where $\alpha_1$ is the lowest non-zero eigenvalue of $\mathcal L_{\rm norm}$, cf.~\cite[Fig.~1]{BelMug13}. And indeed abundant information is available on the spectrum of $\mathcal L_{\rm norm}$, see e.g.~\cite{Chu97,ButChu13}. However, the setting of~\cite{Bel85} is a very special case of general quantum graphs: the topic of this paper will be the far more challenging case of quantum graphs with different edge lengths. One may argue that investigating the spectrum of a differential operator is less convenient than working with a matrix. In fact, we maintain that our approach based on quantum graphs has some advantages that come from the flexibility offered by the continuous setting.

On the other hand, spectral geometry on quantum graphs is markedly different from on domains or manifolds, since elementary variational principles become far more powerful in an essentially one-dimensional setting: here, one can perform various types of ``surgery'' on graphs which have a given effect on the spectral gap. Our analysis will typically be based on nothing more than an -- at times rather subtle -- application of these principles, together with an explicit analysis of the resulting class of extremising graphs to identify the overall maximiser or minimiser. It is thus all the more surprising that so little seems to be known; in fact, one of the messages of the current paper is that one can go much further, and by more elementary means, than on domains or manifolds.

Quantum graphs have already been occasionally used in the past as a tool for spectral investigations of manifolds, cf.~\cite{Col86}; the main goal of this paper is to start, however, a systematic investigation of spectral geometry of quantum graphs. In Sections~\ref{sec:basic} and~\ref{sec:ex} we summarise the elementary properties of quantum graphs which we will need in the sequel: in particular, in Section~\ref{sec:basic}, we state the fundamental variational principles we will use, which show how the spectral gap depends on structural properties of the graph (see Lemma~\ref{lem:principles}); most of these have already appeared scattered throughout the literature, albeit not in one place, and they do not seem to have been used previously to study extremising problems. In Section~\ref{sec:ex} we list a number of classes of graphs with natural extremising properties. Section~\ref{sec:lev} is devoted to proving an upper bound on the spectral gap in terms of $L,E$ that complements the known lower bound~\eqref{eq:l-lower} -- the old result of S.\ Nicaise alluded to above that has been rediscovered several time since~\cite{Nic87}. In Section~\ref{sec:diam} we will prove that fixing the diameter $D$ of a graph alone is not enough to yield estimates on the spectral gap. This is in our opinion the most surprising result of this paper. In particular, in order to show that $D$ alone cannot bound the spectral gap from above we will introduce a special class of graphs, so-called \emph{pumpkin chains}, see Definition~\ref{def:hubgraph}. They have a large spectral gap for given diameter and will allow us to reduce our problem to a Sturm--Liouville one. While one-dimensional reductions have been used for the spectral analysis of (rather particular types of) graphs in the past \cite{HisPos09,Sol04}, our approach seems somewhat different: we do not need to make any symmetry assumptions on our graphs, and we obtain our results by studying sequences of Sturm--Liouville operators with smooth coefficients, that is, we are lead to natural, intrinsically one-dimensional phenomena. This method is also quite possibly amenable to further development. These pumpkin chains will also allow us to prove various upper bounds in conjunction with other quantities such as $V$ or $L$ in Sections~\ref{sec:diam-pos} and \ref{sec:diam-length}. We will briefly summarise the bounds for the normalised Laplacian which can be deduced from our results in Section~\ref{sec:discrete}.

While our graphs will always have finitely many edges and vertices, in some cases we are going to prove bounds that cannot be attained by such finite graphs, but are approximated by suitable families of finite graphs with increasingly many edges.

\medskip
We can represent the outcome of our investigations in a schematic form; two tables in Section~\ref{sec:concl} summarise the corresponding bounds and whether optimal graphs realising the bounds exist:
\medskip
\begin{center}
\begin{tabular}{r|l|l}
		\backslashbox{parameter(s)}{$\lambda_1$} & upper estimate? & lower estimate?  \\ \hline
$V$, $E$ &   \xmark~(Remark~\ref{rem:cannot}) & \xmark~(Remark~\ref{rem:cannot})\\ \hline
$L$ & \xmark~(Eq.~\eqref{eq:l-lower}) & \cmark~(Eq.~\eqref{lambda1flower}) \\ \hline
$L$, $V$ &\xmark~(Eq.~\eqref{lambda1flower}) & \cmark~(Eq.~\eqref{eq:l-lower}, Example~\ref{ex:lev-lower})   \\ \hline
$L$, $E$ & \cmark~(Theorem~\ref{th:le}) & \cmark~(Eq.~\eqref{eq:l-lower}, Example~\ref{ex:lev-lower}) \\ \hline
$D$ & \xmark~(Theorem~\ref{th:d-upper}) & \xmark~(Example~\ref{ex:d-lower}) \\ \hline
$D$, $V$  & \cmark~(Theorem~\ref{th:dv}) & \xmark~(Example~\ref{ex:d-lower}) \\ \hline
$D$, $E$  & \cmark~(Remark~\ref{rem:de}) & \cmark~(Remark~\ref{rem:de}) \\ \hline
$D$, $L$ &  	\cmark~(Theorem~\ref{th:dl-upper}) & \cmark~(Theorem~\ref{th:dl-lower}) 
			\end{tabular}
	\label{tab:resume-2}
\end{center}

\section{Notation and basic techniques} \label{sec:basic}

Throughout this paper all graphs are metric graphs, unless otherwise stated. We shall mostly adopt the usual notation of graph theory: graphs will be denoted by $\mG$, their edge and vertex sets will be denoted by $\mE$ and $\mV$, respectively, then again we will adopt the notation $\me$ and $\mv$ for edges and vertices, respectively. In graph theory it is customary to denote by $n$ (resp., $m$) the cardinality of $\mV$ (resp., $\mE$), but in this case we prefer to adopt the alternative notation
\[
V(\mG):=|\mV|  \qquad \hbox{and}\qquad E(\mG):=|\mE|,
\] 
since we wish to perform analysis on graphs. We will denote by 
\[
L(\mG) := \sum_{\me\in\mE}|\me|
\]
the total length of $\mG$, i.e., the sum of the lengths of all edges of $\mG$, and the diameter of $\mG$ by
\begin{equation}
\label{eq:d}
	D(\mG):=\sup\left\{\dist\,(x,y):x,y\in\mG\right\},
\end{equation}
where the distance between two points of a graph is as usual defined to be the length of the shortest path within $\mG$ connecting them (cf.~\cite[\S~3.2]{Mug14}). Note that here we take the supremum over \emph{all} points $x,y\in \mG$. If we restrict ourselves to considering $x,y\in\mV$, then we have a ``weaker'' notion of diameter, which we shall call the \emph{combinatorial diameter} $D_\mV$ of $\mG$, i.e.
\begin{equation}
\label{eq:comb-dv}
	D_\mV(\mG):=\sup\left\{\dist\,(x,y):x,y\in V(\mG)\right\}.
\end{equation}
This is consistent with the classical notion of diameter from combinatorial graph theory, and is in practice easier to compute if the graph is particularly large. Obviously, we have $D_\mV(\mG) \leq D(\mG)$, while if the longest edge in $\mG$ is of length $a>0$, say, then since the maximum distance of any point of $\mG$ to $\mV$ is $a/2$, we have $D(\mG) \leq D_\mV(\mG)+a$. See also Remark~\ref{rem:vertex-d}.

We will also impose the following standing assumption on all the graphs we consider to ensure simultaneously the non-triviality and the finiteness of all the quantities we consider.
\begin{assumption}
\label{ass:graph}
The metric graph $\mG$ is connected. It is compact and finite, i.e., it consists of finitely many edges of finite length.
\end{assumption}

The object of our investigations is a realisation of the operator that acts as a second derivative on the intervals associated with each edge of $\mG$. The most common choice in the literature is to define a Laplacian $\Delta$ on a quantum graph by imposing two conditions in each vertex of $\mG$: functions in the domain of $\Delta$
\begin{itemize} 
\item are continuous across the vertices and
\item their normal derivatives about each vertex sum up to 0.
\end{itemize}
Under Assumption~\ref{ass:graph} and with this choice of ``natural'' (i.e.~continuity and Kirchhoff) boundary conditions, $-\Delta$ is a self-adjoint, positive semi-definite operator with compact resolvent on the Hilbert space $L^2(\mG)$ of square integrable functions supported on the intervals associated with the edges of the graph. Thus, $-\Delta$ has pure point spectrum $\sigma(-\Delta)\subset [0,\infty)$; we will call its elements simply the eigenvalues of $\mG$, since the operator is uniquely determined by the metric graph $ \mG$. 

The value $ \lambda_0 = 0 $ is an eigenvalue of multiplicity one (since $\mG$ is connected) with eigenfunction $ u_0 \equiv 1. $ Thus the spectral gap coincides with the lowest non-trivial eigenvalue $ \lambda_1 > 0, $
which can be obtained by minimising the Rayleigh quotient subject to the constraint of $L^2$-orthogonality to the eigensubspace corresponding to $\lambda_0 =0$:
\begin{equation}
\label{eq:rq}
	\lambda_1(\mG) = \inf\left\{ \frac{\int_\mG |u'(x)|^2\,\textrm{d}x}{\int_\mG |u(x)|^2\,\textrm{d}x} :
	u \in H^1(\mG),\, {\int_\mG u(x)\,\textrm{d}x}=0 \right\}\ .
\end{equation}
Here $H^1(\mG)$ -- the form domain of $-\Delta$ -- is the space consisting of those functions defined on the intervals associated with the edges of the graph, belonging to the first Sobolev space $H^1$ thereon, and satisfying continuity conditions in the vertices, cf.~\cite[\S~1.3]{BerKuc13} or~\cite[\S~3.2]{Mug14}. 

\begin{remark}
(a) Unlike in combinatorial graph theory, for the purposes of quantum graph theory vertices of degree two are unessential objects that can be inserted or removed without changing either the space $H^1(\mG)$ or the domain of $\Delta$, and in particular without affecting the spectrum.

(b) In general, we are not going to assume graphs to be simple. Indeed, for our purposes it is always possible to add dummy vertices in the middle of an edge $\me$ -- thus replacing an edge of length $|\me|$ by two edges of length $|\tilde{\me}|$ and $|\me|-|\tilde{\me}|$, respectively -- turning a non-simple graph into a simple one without changing the spectrum of the Laplacian.
\end{remark}

From \eqref{eq:rq} follow a handful of elementary but powerful principles of which we will make extensive use; we note that many of these have already appeared in \cite{KurMalNab13}, see also \cite{BerKuc12,ExnJex12} for other results on edge dependence of eigenvalues and eigenfunctions of a slightly different flavour.
\begin{lemma}
\label{lem:principles}
Suppose $\mG$ and $\mG'$ are quantum graphs satisfying Assumption~\ref{ass:graph}.
\begin{enumerate}
\item If $\mG'$ is formed by attaching a pendant edge, or more generally a pendant graph, to one vertex of $\mG$, then $\lambda_1(\mG) \geq \lambda_1(\mG')$.
\item If $\mG'$ is formed from $\mG$ by identifying two vertices of $\mG$ (say, $\mv_1,\mv_2$ are replaced with a new vertex $\mv_0$ and each edge having $\mv_1$ or $\mv_2$ as an endpoint is replaced with a new edge having $\mv_0$ as an endpoint, and in particular the edges between $\mv_1$ and $\mv_2$ are replaced with loops around $\mv_0$), then $\lambda_1(\mG)\leq \lambda_1(\mG')$.
\item If we add an edge $\me=\mv_1\mv_2$ between two already existing vertices of a quantum graph $\mG$, producing the new quantum graph $\mG'$, then $\lambda_1(\mG)\ge  \lambda_1(\mG')$ \emph{provided} 
there is an eigenfunction corresponding to 
$\lambda_1(\mG)$ attaining the same value on both $\mv_1,\mv_2$.
\item If $\mG'$ is formed from $\mG$ by lengthening a given edge, then $\lambda_1(\mG) \geq \lambda_1(\mG')$.
\item If $ \mG' $ is obtained from $ \mG $ by scaling each edge with the factor $ 1/c \in \mathbb R $, then the corresponding eigenvalues scale as $ c^2 $, that is,
$$ \lambda_1 (\mG) = c^{-2} \lambda_1 (\mG'). $$
\end{enumerate}
\end{lemma}

By ``pendant graph'' (or edge) we mean that the graph to be added, i.e. $\mG' \setminus \mG$, is attached to $\mG$ only at the one vertex, sometimes called a ``cutvertex'' in the graph theoretical literature. In particular, this covers the case of adding a loop.

\begin{proof}[Proof of Lemma~\ref{lem:principles}] We will merely sketch a proof of the statements; details can in most cases be found in \cite{KurMalNab13}.

\noindent (1) Let $\psi$ be any eigenfunction associated with $\lambda_1(\mG)$, and suppose that we obtain $ \mG'$ by attaching a pendant graph to a vertex $ \mv \in \mG.$
 We extend $\psi$ to a function $\tilde\psi\in H^1(\mG')$ by setting $\tilde\psi=\psi(\mv)$ on $\mG' \setminus \mG $. Then the function $\varphi = \tilde\psi - \int_{\mG' \setminus \mG}\tilde\psi$ is a valid test function for $\lambda_1(\mG')$, and it may be checked that the Rayleigh quotient of $\varphi$ is not larger than $\lambda_1(\mG)$. 

\noindent (2)  This follows immediately from the fact that $H^1(\mG')$ may be identified with a subspace of $H^1(\mG)$, since the continuity conditions imposed on functions in the former space are more restrictive; but the Rayleigh quotient is given by the same formula.

\noindent (3) The proof is similar to (1): extending the eigenfunction $ \psi_1 $ minimising the Rayleigh quotient for $ \mG $ by a constant on the new edge (equal to the common value $ \psi_1 (\mv_1) =
\psi_1 (\mv_2) $) yields a trial function for $ \mG'$ whose Rayleigh quotient can be no larger, even after a possible orthogonalisation. 

\noindent (4) Let us denote by $ \me_1 \subset \mG $ the edge to be lengthened, which we identify with the interval $(0,a)$, identify the lengthened interval with $(0,a')$, $a'>a$, and let $ \psi_1 $ be the eigenfunction on $ \mG$. Consider the function $ \tilde\psi \in H^1(\mG')$ which is equal to $ \psi_1 $ on $\mG \setminus \me_1$ (identified canonically with a subset of $\mG'$) and on $(0,a)$, and which is extended by $\psi_1(a)$ on $(a,a')$. As in (1), the function $\varphi = \tilde\psi - \int_{\mG'} \tilde \psi = \tilde \psi - (a'-a)\psi(a)$ is a valid test function for $\lambda_1 (\mG')$, and as in (1), its Rayleigh quotient is not larger than $\lambda_1 (\mG)$. Alternatively, this may be viewed as a special case of (3).

\noindent (5) The proof follows essentially (4). One needs to take into account that all edges are scaled with the same factor and no orthogonalisation is needed.
\end{proof}

\begin{remark}\label{rem:cannot}
(a) Statement (3) can be interpreted as saying that diffusion processes in quantum graphs may in some cases actually converge to equilibrium faster upon \emph{removing} edges. This seeming paradox seems to have been first explicitly remarked on  in~\cite{KurMalNab13} and may be solved by realising that the edges whose removal increases the spectral gap are in some way redundant.

(b) Statement (2) says that ``pulling apart'' any vertex of a graph (to create a new graph with the same set of edges and lower connectivity) always lowers the spectral gap. This process can be continued until a quantum graph of minimal connectedness -- a tree -- is reached. This seems reminiscent of the principle that each spanning tree of a combinatorial graph $\mG$ has lower spectral gap (for the discrete Laplacian) than $\mG$ itself, cf.~\cite[Cor.~3.4]{Moh91}. Obviously, the metric trees associated in this way with $\mG$ are formed from the same set of edges as $\mG$, while the (discrete) spanning trees live on the same set of vertices as $\mG$. It is an interesting question, but one we will not investigate here, to what extent these metric trees really are the metric equivalent of ``spanning trees''.

(c) On combinatorial graphs $\mG$, there is a wide range of upper and lower bounds available which depend only on intrinsically combinatorial quantities such as $V$, the minimal degree $\deg_{\min{}}(\mG)$, the maximal degree $\deg_{\max{}}(\mG)$ and the edge connectivity $e(\mG)$ of $\mG$. We mention for example that the second lowest eigenvalue $\beta_1$ of the discrete Laplacian $\mathcal L$ on a combinatorial graph satisfies the bounds 
\[
2e(\mG)\left(1-\cos\frac{\pi}{V}\right) \leq \beta_1 \leq \frac{V}{V-1}\deg_{\min{}}(\mG)\ ,
\]
cf.~\cite{Fie73}, whence the bounds
\begin{equation}\label{eq:spielman-revis}
\frac{2e(\mG)}{\deg_{\max{}}(\mG)} \left(1-\cos\frac{\pi}{V}\right) \leq \alpha_1 \leq \frac{V}{V-1}\ ,
\end{equation}
on the second lowest eigenvalue $\alpha_1$ of the normalised Laplacian $\mathcal L_{\rm norm}$; if $\mG$ is planar, then the upper bounds may be improved to
\[
\beta_1\le \frac{8}{V}\deg_{\max{}}(\mG)\qquad \hbox{and}\qquad \alpha_1\le \frac{8}{V}\frac{\deg_{\max{}}(\mG)}{\deg_{\min{}}(\mG)}\ ,
\]
cf.~\cite{SpiTen07} (see also Section~\ref{sec:discrete} for definitions and references to further bounds). However, Statement (5) implies the obvious fact that spectral estimates for the continuous Laplacian cannot be obtained without taking into account any metric parameter like the total length or the length of one of the edges. In particular estimates based solely on the number of vertices and/or edges are impossible. This also confirms that existing estimates for combinatorial Laplacians cannot be directly applied in our case, since combinatorial graphs lack metric parameters.

(d) It is natural to ask whether the statements also hold for the higher eigenvalues. It is obvious (and trivial to adapt the proofs to show) that (2) and (5) hold for all $\lambda_n$, $n \in \N$. For the others this is far less clear; (1) and (4) seem intuitively obvious, but the natural proof for $\lambda_1$ does not seem to generalise easily owing to the more complicated nature of the orthogonality conditions that test functions for the higher eigenvalues need to satisfy. We do not expect (3) to hold in general, although we do not have a counterexample. However, as this is not relevant for the analysis at hand, we will not pursue it further here.
\end{remark}

\section{Extremal graphs and special classes of graphs} \label{sec:ex}

In the process of studying optimisation problems, it is natural to look for classes of potential optimisers, and the following special classes of graphs seem at various times to play an important role.
In order to obtain effective estimates it is important to be able to calculate the spectrum of an optimiser explicitly, which is possible only for graphs with very special choice of the edge lengths.
If all edges of a quantum graph have the same length, then we call it \emph{equilateral}: unless otherwise stated, we will always assume flower, star, pumpkin, and complete graphs to be equilateral.

\begin{itemize}

\item \emph{Path} graphs $ \mI(L)$ are intervals of length $L$ viewed as quantum graphs.
Because on a path graph the imposed boundary conditions boil down to continuity of a function and its derivative on the internal vertices and Neumann on the extremal ones, 
the spectral gap of a path graph of total length $L$ clearly agrees with that of an interval of length $L$, i.e.,
\begin{equation}\label{lambda1path}
\lambda_1 (\mI(L)) =\frac{\pi^2}{L^2}\ .
\end{equation}

\item \emph{Flower} graphs $ \mF (L, E) $ are quantum graphs built upon non-simple combinatorial graphs consisting of one vertex and $E$ loops attached to it, which we shall call \emph{petals}.
\begin{figure}
\begin{tikzpicture}[scale=0.7]
  \begin{polaraxis}[grid=none, axis lines=none]
     \addplot[mark=none,domain=0:360,samples=300] { abs(cos(7*x/2))};
   \end{polaraxis}
 \end{tikzpicture}
\caption{A flower graph $\mF(L,7)$ on seven edges (petals).}
\end{figure}
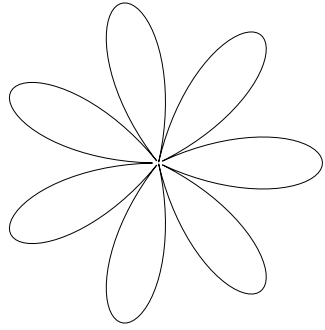
A direct computation shows that the spectral gap of a flower graph of total length $L$ on $E \geq 2$ edges is
\begin{equation}\label{lambda1flower}
\lambda_1 (\mF (L, E)) =\frac{\pi^2 E^2}{L^2}\ ;
\end{equation}
if $E=1$, then we have a loop of length $L$, which has $\lambda_1 = 4\pi^2/L^2$.

\item \emph{Star} graphs $ \mS (L,E) $ are quantum graphs consisting of a central vertex $\mv$ and $E \geq 2$ edges radiating out from $\mv$. If we take all edges to have equal length $L/E$ (as we shall always do), then
\begin{equation}\label{lambda1star}
	\lambda_1(\mS (L,E)) = \frac{\pi^2 E^2}{4L^2}.
\end{equation}

\item \emph{Pumpkin} graphs $ \mP(L,E) $, sometimes also called \emph{dipole} or \emph{banana} graphs in the literature, are quantum graphs built upon non-simple combinatorial graphs consisting of two vertices and $E$ parallel edges, which we shall call \emph{slices}, having both vertices as endpoints. The spectral gap of a pumpkin graph of total length $L$ on $E$ edges is also
\begin{equation}\label{lambda1pumpkin}
\lambda_1 (\mP (L, E)) =\frac{\pi^2 E^2}{L^2}\ ,
\end{equation}
since one can see directly that there is a corresponding eigenfunction having the form $\sin(\frac{\pi E\cdot}{L})$ on each edge, where we identify each edge with the interval $(0,\frac{L}{E})$.
\begin{figure}
\begin{tikzpicture}[scale=0.8]
\coordinate (A) at (0,0);
\coordinate (D) at (6,0);
\draw (A) -- (D);
\draw (A) to [bend left] (D);
\draw (A) to [bend left=60] (D);
\draw (A) to [bend left=110] (D);
\draw (A) to [bend right] (D);
\draw (A) to [bend right=60] (D);
\draw (A) to [bend right=110] (D);
\end{tikzpicture}
\caption{A pumpkin graph $\mP(L,7)$ on seven edges (slices).}
\end{figure}
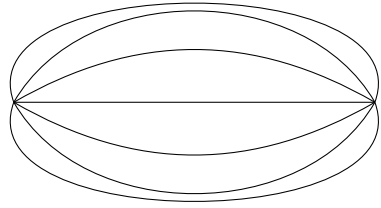

\item \emph{Complete} graphs $ \mK_{V}$ are quantum graphs built upon simple combinatorial graphs consisting of $V$ vertices and exactly one edge joining any pair of vertices, meaning $E=\frac{V(V-1)}{2}$ edges in total. By~\eqref{eq:below} the spectral gap of a complete graph of total length $L$ on $E$ edges is
\begin{equation}\label{lambda1complete}
\lambda_1 (\mK_V) =\left(\arccos\frac{1}{1-V} \right)^2 \frac{V^2 (V-1)^2}{4 L^2}\ .
\end{equation}
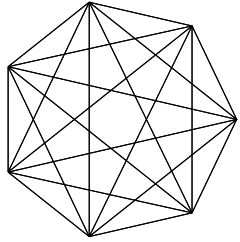
\begin{figure}
\begin{tikzpicture}[scale=0.8]
\foreach \x in {51, 103, 154, 206, 257, 309, 360}{
\foreach \y in {51, 103, 154, 206, 257, 309, 360}{
\draw (\x:2cm) -- (\y:2cm);
}}
\end{tikzpicture}
\caption{A complete graph $ \mK_7 $ on seven vertices.}
\end{figure}
\end{itemize}

\begin{remark}
The expressions in~\eqref{lambda1flower} and~\eqref{lambda1pumpkin} imply in particular that adding to an (equilateral) flower or pumpkin graph arbitrarily many petals or slices (in the case of flowers and pumpkins, respectively) does not change the spectral gap,  instead increasing the multiplicity of $\lambda_1$, as long as each of the new edges has the same length as the original ones. 
\end{remark}

In several cases we will also perform graph surgery to produce examples. A distinguished class consists of what we call \emph{dumbbell graphs}, which are obtained by joining two quantum graphs (ideally, two ``massive ones'') by only one edge. Naturally, dumbbell graphs are the analogue of dumbbell domains.

Remarkably, we are going to see that in most cases the correct quantum graph equivalents of balls are not complete graphs, as one may naively think (and -- as a rule of thumb -- as is actually the case when \emph{discrete} graphs are investigated~\cite{Fie73}), but rather flower graphs.

\section{Estimates involving the length} \label{sec:lev}

Perhaps the most natural quantity of a graph $\mG$ to consider is its total length $L$, which is the equivalent of the volume of a domain. However, quantum graphs lack a natural (continuous) notion of perimeter, and in some sense $V$ and $E$ might be thought of as discrete counterparts thereof; so it is also natural to consider $V$ and/or $E$ in combination with $L$. This section is thus devoted to estimates on the spectral gap in terms of the total length $ L $ and possibly $ V $ and $ E$; first we will consider lower bounds, and then upper bounds.

We remark at this juncture that a graph on $E$ edges can have at most $V=E+1$ vertices, so that bounding $E$ from above automatically also bounds $V$ from above; put differently, controlling $E$ is in a certain sense a stronger restriction on the graph than controlling $V$. Moreover, using standard matching conditions allows one to increase the number of vertices and edges without actually changing the metric graph (by introducing new vertices at inner points on the edges), so in general we expect controlling $E$ and $V$ from above to lead to both better upper \emph{and} lower bounds, since this is a restriction on the degree of complexity of the graph.

\subsection{Lower estimates}
 
The first and fundamental lower estimate was obtained by S.\ Nicaise~\cite[Théorème 3.1]{Nic87} and later independently re-proved by L.\ Friedlander \cite[Theorem~1]{Fri05} and by S.\ Naboko and one of the present authors \cite[Theorem~1]{KurNab14}:
\begin{equation}
\label{eq:l-lower}
\lambda_1(\mG) \geq \frac{\pi^2}{L^2}\ ,
\end{equation}
where $\mG$ is any graph having total length $L>0$. The minimiser is essentially unique: it is simply a path graph, which is a one-dimensional ball. This estimate is easy to understand intuitively: the diffusion in a quantum graphs of total length $ L $ can be at most as slow as on the interval $[0,L]$.

The lower estimate cannot be improved without taking into account new parameters, since we have a minimiser. However, neither $E$ nor $V$ can be used in combination with $L$ to give a refined lower bound on $\lambda_1(\mG)$ any better than \eqref{eq:l-lower}, as we show next. The idea behind this is that perturbing a graph in a small but rather complicated way does not affect the spectral gap in a serious way, although it can have an arbitrarily large effect on the graph's combinatorics.

\begin{example}
\label{ex:lev-lower}
Given any natural numbers $E \geq 2$ and $V \geq 2$ with $V \leq E + 1$ and any positive numbers $L>0$ and $\varepsilon>0$, we can find a connected graph $\mG$ having $E$ edges, $V$ vertices and total length $L$, such that
\begin{displaymath}
	\lambda_1(\mG) \leq \frac{\pi^2}{L^2} + \varepsilon\ .
\end{displaymath}
To do so, we take an arbitrary connected graph $\tilde\mG$ having $V-1$ vertices and $E-1$ edges (if $V\geq 3$ this is always possible to do by starting with a tree with $V-1$ vertices and $V-2 \leq E-1$ edges, and adding extra edges between vertices in an arbitrary fashion until there are $E-1$; if $V=2$, we take a flower with $E-1$ petals), rescale $\tilde\mG$ to have total length $\delta>0$ and form $\mG$ by taking a $V^{\rm th}$ vertex $\mv$ and joining $\mv$ to an arbitrary vertex $\mw$ of $\tilde\mG$ by an $E^{\rm th}$ edge $\me = \mv\mw$ of length $L-\delta$. It follows from Lemma~\ref{lem:principles}.(1) that $\lambda_1(\mG) \leq \lambda_1(\me) = \pi^2/(L-\delta)^2$ (``adding'' the graph $\tilde\mG$ to the edge $\me$ can only decrease the eigenvalue); choosing $\delta>0$ small enough proves the claim.
$\blacksquare$\end{example}

\subsection{Upper estimates.}

The explicit examples considered in Section~\ref{sec:ex} show that no upper estimate is possible in terms of $L$ alone: by~\eqref{lambda1flower}, $\lambda_1$ can be made arbitrarily large by considering flower graphs on more and more edges of smaller and smaller length. Moreover, the same example shows that fixing just $L$ and $V$ together is not enough to bound $\lambda_1$ from above.

However, if we fix $L$ and $E$, then we obtain a sharp upper bound on $\lambda_1$ in terms of the arithmetic mean value $ \mathcal A = L/E $ of the edge length.

\begin{theorem}
\label{th:le}
Let $\mG$ be a quantum graph having length $L>0$ and $E \geq 2$ edges. Then
\begin{equation} \label{upperest}
	\lambda_1(\mG)\leq \frac{\pi^2 E^2}{L^2} \equiv \frac{\pi^2}{\mathcal A^2},
\end{equation}
with equality if and only if $\mG$ is an equilateral pumpkin or flower graph.
 If $E=1$, then  
 $$ \lambda_1 (\mG) = \left\{
 \begin{array}{ll}
  4 \pi^2/L^2 & \mbox{if $ \mG $ is a loop,} \\
   \pi^2/L^2 & \mbox{if $ \mG $ is a path.}
   \end{array} \right. $$ 
\end{theorem}

It is an interesting phenomenon that we have two fundamentally distinct classes of maximisers. In fact, our proof implicitly contains a second characterisation of the maximisers, which curiously allows us to obtain indirectly the following statement.

\begin{corollary}
\label{cor:max}
A finite equilateral quantum graph $\mG$ of length $L>0$ and $E \geq 3$ edges admits an eigenfunction corresponding to $\lambda_1(\mG)$ which takes on the value zero at all vertices of $\mG$ if and only if $\mG$ is an (equilateral) pumpkin or flower graph.
\end{corollary}

\begin{proof}[Proof of Theorem~\ref{th:le}]
Assume $\mG$ is any quantum graph as in the statement of the theorem, with $L>0$ and $E \geq 2$ given, and denote by $\tilde\mG$ the corresponding 
flower graph having the same number of edges with the same lengths as $\mG$, i.e., $\tilde\mG$ is the graph which may be formally obtained from $\mG$ 
by identifying all vertices of the latter. Then $H^1(\tilde\mG)$ may be canonically identified with the subset of $H^1(\mG)$ consisting of all functions 
$u \in H^1(\mG)$ such that $u(\mv_1)=\ldots=u(\mv_n)$ for all vertices $\mv_1,\ldots,\mv_n$ of $\mG$. It follows from Lemma~\ref{lem:principles}.(2) that $\lambda_1(\mG) \leq \lambda_1(\tilde\mG)$.

We now show that the equilateral flower graph $\mF(L,E) $ is the (unique) maximiser of $\lambda_1$ among all flower graphs of fixed total length $L>0$ and number 
of edges $E\geq 2$, which will then complete the proof of (\ref{upperest}), since $\lambda_1(\mF(L,E)) = \pi^2 E^2/L^2$ by~\eqref{lambda1flower}. To that end, let $\me_1$, $\me_2$ be the 
longest two edges of the arbitrary flower graph $\tilde\mG$ (or any two longest edges if these are not uniquely determined); by the pigeonhole principle, 
$M:=|\me_1|+|\me_2| \geq 2L/E$. Denote by $\mG_{12} $ the flower graph consisting of these two longest edges. Since each petal of $\tilde\mG$ may be regarded as a pendant graph attached to $\mG_{12}$, by Lemma~\ref{lem:principles}.(1) (or Lemma~\ref{lem:principles}.(3)),
\begin{displaymath}
	\lambda_1(\tilde\mG) \leq \lambda_1(\mG_{12})\ .
\end{displaymath}
But we see immediately that $\lambda_1(\mG_{12}) \leq 4\pi^2/M^2 \leq \pi^2 E^2/L^2$, since we may use any eigenfunction belonging to the first eigenvalue of a 
circle of length $M$ as a test function on $\mG_{12} $, provided it is rotated appropriately so as to satisfy the continuity condition at the points corresponding to the vertex of $\mG_{12}$. 
This establishes the inequality.

To prove the case of equality, we first note that this is only possible if $M=2L/E$, that is, if the two longest edges (and thus all other edges) have length $L/E$ each, and so $\tilde\mG$ must already be equilateral in this case.

We also note that pumpkins and flowers satisfy the equality, cf.~\eqref{lambda1pumpkin} and \eqref{lambda1flower}. We now show via a test function argument that any equilateral graph $\mG$ with at least three vertices must have $\lambda_1(\mG) < \pi^2 E^2/L^2$; the cases $V=1,2$ (and $E=1$) are trivial. To that end, suppose $\mG$ is equilateral with $V(\mG)\geq 3$. Then there exists a partition of $\mG$ into graphs $\mG_+, \mG_-, \mG_0$ with the following properties: $\mV(\mG_+) \cap \mV(\mG_-) = \emptyset$, every edge in $\mG_0$ begins at a vertex in $\mG_+$ and ends at a vertex in $\mG_-$, and $\mE(G_+)$ is non-empty (but it is possible that $\mG_-$ consists of a single vertex).

We now construct a test function $\psi$ on $\mG$ by setting $\psi|_{\mG_+} := 1$, $\psi|_{\mG_-} := -1$, and on each edge of $\mG_0$, which we identify with the interval $[0, L/E]$ (the point $0$ corresponding to a vertex in $\mG_+$ and $L/E$ to a vertex in $\mG_-$), we let $\psi(x)=\cos(\pi E x/L)$. Then $\psi \in H^1(\mG)$, and by construction $\psi|_{\mG_0}$ has Rayleigh quotient equal to $\pi^2 E^2/L^2$. Although $\psi$ will not necessarily have mean value zero on $\mG$, since $\psi$ is constant on $\mG \setminus \mG_0$, the Rayleigh quotient of $\psi - \int_\mG \psi$ on $\mG$ can only be lower than that of $\psi$ on $\mG_0$; the proof of this claim is identical to the proof of Lemma~\ref{lem:principles}.(1) or (4). Moreover, since $\mG \setminus \mG_0$ has positive measure, $\psi - \int_\mG \psi$ cannot be an eigenfunction on $\mG$; hence $\lambda_1(\mG)$ is strictly smaller than the Rayleigh quotient of $\psi-\int_\mG\psi$, which is no larger than $\pi^2 E^2 /L^2$.
\end{proof}

\begin{proof}[Proof of Corollary~\ref{cor:max}]
The ``if'' statement is clear. For the other direction, we show that the equality $\lambda_1(\mG) = \pi^2 E^2/L^2$ follows from the existence of an eigenfunction $\psi$ of $\mG$ vanishing on all vertices (in fact this is equivalent, that is, a \emph{characterisation}, but we do not need this). Since by Theorem~\ref{th:le} only pumpkins and flowers have this property, this establishes the claim. If $\psi(\mv_1)=\ldots = \psi(\mv_n) = 0$, then, by identifying all vertices of $\mG$, $\psi$ can be mapped directly onto a function $\tilde\psi \in H^1(\mF)$ having mean value zero and the same Rayleigh quotient, where $\mF$ is a flower graph having the same length and number of edges as $\mG$.

Since $\psi$ satisfies the equation $-u'' = \lambda_1(\mG) u$ on each edge of $\mG$ and the Kirchhoff condition at each vertex of $\mG$, the same is necessarily true of $\tilde\psi$ on $\mF$ (for the vertex condition, sum up \emph{all} the derivatives of $\psi$ on \emph{all} vertices of $\mG$; this is still zero). Hence $\tilde\psi$ is an eigenfunction on $\mF$ with eigenvalue $\lambda_1(\mG)$. Since $\lambda_1(\mG) \leq \lambda_1(\mF)$ by Lemma~\ref{lem:principles}.(2), and $\tilde\psi$ has the same Rayleigh quotient as $\psi$, the only possibility is that $\tilde\psi$ is associated with $\lambda_1(\mF)$. Hence there is equality.
\end{proof}

\begin{remark}
Theorem~\ref{th:le} can be regarded as a counterpart to the value of the spectral gap of the second derivative with Neumann conditions on an interval, given that $L/E$ is the arithmetic mean of all edge lengths. A similar assertion for another problem, the first eigenvalue of the Laplacian with a Dirichlet condition on at least one vertex, was recently proved in~\cite[Thm.~3.8]{DelRos14}. There are other, weaker estimates on $\lambda_1$ in terms of the arithmetic, geometric and harmonic means of the sequence of edge lengths of the graph, which we denote by ${\mathcal A},{\mathcal G},{\mathcal H}$, respectively: \cite[Thm.~2]{Kur15} states that
\[
	\lambda_1(\mG) \leq \frac{4\pi^2}{\mathcal{AH}}=\frac{4\pi^2}{\mathcal{G}^2};
\]
the factor of four may even be dropped if $\mG$ is bipartite. Observe that this estimate is still sharp, as it is attained in the case of (equilateral) flowers and pumpkins, although \eqref{upperest} is better, including in the case where $ \mG $ is bipartite.

An estimate of a similar flavour in terms of the maximal edge length $M$ of $\mG$ can be obtained via a trivial test function argument, by placing a suitable sine function on the longest edge, extended elsewhere by zero: $\lambda_1(\mG) \leq 4\pi^2/M^2$. This is however in general far worse than \eqref{upperest}. This may be the natural counterpart to the usual estimate in terms of in-radii in the case of domains.
\end{remark}

\section{Estimates in terms of the diameter alone: negative results} \label{sec:diam}

Another natural choice of metric quantity is the graph diameter $D$ (see \eqref{eq:d}). Diameter has been regarded as a natural quantity in spectral optimisation of (convex) domains at least since van den Berg's formulation of the Gap Conjecture in 1983, finally proved by Andrews and Clutterbuck just a few years ago \cite{AndClu11}, which in the case of the Dirichlet Laplacian states that the difference between the first two eigenvalues on a convex $d$-dimensional domain of diameter $D$ is at least $3\pi^2/D^2$, the corresponding value for an interval of length $D$. Diameter also plays an interesting role in spectral gap estimates for combinatorial graphs, often in interplay with the somewhat related quantity of mean distance, cf.~\cite{AloMil85,Moh91b,Chu97,Mol12}.

In this section we will prove that fixing the diameter of a graph alone is not in fact enough to yield upper or lower estimates on the spectral gap. Finding examples showing the impossibility of lower bounds is fairly straightforward; the question of an upper bound is however far more subtle. To that end we will introduce the \emph{pumpkin chains} referred to in the introduction (see Definition~\ref{def:hubgraph} below), which are an extremising class of graphs for diameter (Lemma~\ref{lem:reduction}). We will not need this extremising property in order to prove that no estimate is possible, but we will use it in Sections~\ref{sec:diam-pos} and \ref{sec:diam-length} to obtain various upper bounds in conjunction with other quantities such as $V$ or $L$. We also present it in this section to try to give a better insight into how to generate graphs with large spectral gaps. The same types of examples will allow us to show certain other quantities alone also do not control the spectral gap; see Sec.~\ref{sec:rad-cycle}.

\subsection{Lower bounds are impossible}

We will start with the easier case: proving that no lower bound on $\lambda_1$ is possible in terms of the diameter alone. 
This is in contrast with the theory of Laplacians on combinatorial graphs, whose lowest non-trivial eigenvalue is not less than $\frac{4}{VD}$, see e.g.~\cite[\S~5.3]{Mol12}.

\begin{example}
\label{ex:d-lower}
Given $D>0$, we construct a sequence of (finite and connected) dumbbell graphs $\mG_n$, all having diameter $D$, such that $\lambda_1(\mG_n) \to 0$: all these graphs will be flower-dumbbells.

More precisely, we start with an edge $\me_0$ of length $D/2$ with endpoints $\mv_{-1} \neq \mv_1$ and attach to each of these vertices one flower on $n$ edges $\me_{\pm 1},\ldots,\me_{\pm n}$ each of length $D/2$, so that the diameter of $\mG_n$ is $D$ for all $n\in \mathbb N$. Then of course the number of edges of $\mG_n$ becomes arbitrarily large as $n\to \infty$.

However, $\lambda_1(\mG_n) \to 0$: to see this we use a simple test function argument. Identifying $\me_0$ with the interval $(-D/4,D/4)$, we set 
\begin{displaymath}
\psi_n(z):=
\begin{cases}
1 \qquad &\text{if $z \in \me_k, \ 1\le k\le n$}\ ,\\
\sin(2\pi x/D) \qquad &\text{if $z \in (-D/4,D/4)\simeq \me_0$}\ ,\\
-1 \qquad &\text{if $z \in \me_k, \ -n\le k\le -1$}\ .
\end{cases}
\end{displaymath}
Then $\psi_n \in H^1(\mG_n)$ and by symmetry $\int_{\mG_n} \psi_n = 0$; moreover, $\int_{\mG_n} |\psi_n'|^2$ is constant in $n\geq 1$, $\psi_n'$ being supported on the fixed edge $\me_0$, while clearly $\int_{\mG_n} |\psi_n|^2 \to \infty$ as $n \to \infty$. It follows immediately from the variational characterisation \eqref{eq:rq} that $\lambda_1(\mG_n) \to 0$ as $n \to \infty$. $\blacksquare$ \end{example}

\begin{remark}
\label{rem:combs}
Making our graphs \emph{flower} dumbbells allows us to keep the number of vertices of $\mG_n$ equal to two. One can of course attach more general graphs of small diameter but large total length to achieve the same result. One can also find fundamentally different (i.e.~non-dumbbell) types of graphs; for example, consider a sequence of ``comb graphs'' as in Figure~\ref{fig:comb}.
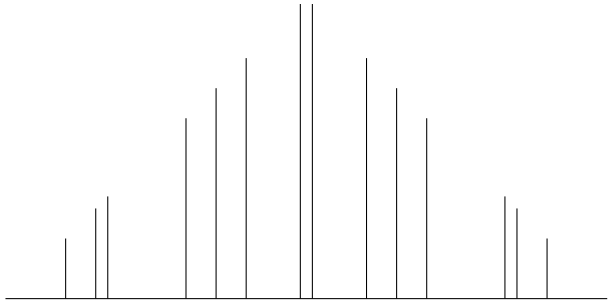
\begin{figure}
\begin{tikzpicture}[scale=0.8]
\draw (-5,0) -- (5,0);
\draw (0.1,0) -- (0.1,4.9);
\draw (1,0) -- (1,4);
\draw (1.5,0) -- (1.5,3.5);
\draw (2,0) -- (2,3);
\draw (3.3,0) -- (3.3,1.7);
\draw (3.5,0) -- (3.5,1.5);
\draw (4,0) -- (4,1);
\draw (-0.1,0) -- (-0.1,4.9);
\draw (-1,0) -- (-1,4);
\draw (-1.5,0) -- (-1.5,3.5);
\draw (-2,0) -- (-2,3);
\draw (-3.3,0) -- (-3.3,1.7);
\draw (-3.5,0) -- (-3.5,1.5);
\draw (-4,0) -- (-4,1);
\end{tikzpicture}\label{fig:comb}
\caption{A comb graph on fourteen prongs.}
\end{figure}
We take a sequence of graphs by adding more and more ``prongs'' (the vertical edges) to the comb, keeping it symmetric and making sure that its diameter remains constant. We then build a test function by constructing a stretched cosine along the lower edge, and extending it by a constant along the prongs; by symmetry, this function is orthogonal to the constant functions and it is easy to check that the corresponding Rayleigh quotient goes to zero provided the prongs do not accumulate too rapidly in the middle.
\end{remark}

\subsection{Upper bounds are also impossible}

To produce a counterexample we follow the natural line of reasoning one would use to attempt to prove that an upper bound exists. We first show that searching of a graph with the largest spectral gap it is enough to consider so-called pumpkin chains (defined below). Then we prove that spectral problem for a pumpkin chain is equivalent to a Sturm--Liouville one-dimensional problem with an integer weight. This analogy allows us to construct a sequence of pumpkin chains of fixed diameter with arbitrarily large gap.

\subsection*{Pumpkin chains maximise the spectral gap.}

The following class of finite graphs is going to play an important role in our constructions:

\begin{definition}
\label{def:hubgraph}
We call a metric graph $\mG$ on $ V $ vertices satisfying Assumption~\ref{ass:graph} a \emph{pumpkin chain} if
\begin{displaymath}
\begin{aligned}
	\# \{ \mv \in \mV: \mv \text{ is adjacent to exactly two other vertices} \} &= V-2,\\
	\# \{ \mv \in \mV: \mv \text{ is adjacent to exactly one other vertex} \} &= 2,
\end{aligned}
\end{displaymath}
and if additionally the lengths of any two edges connecting $\mv,\mw$ is the same for any pair of adjacent vertices $\mv,\mw$. 
\end{definition}

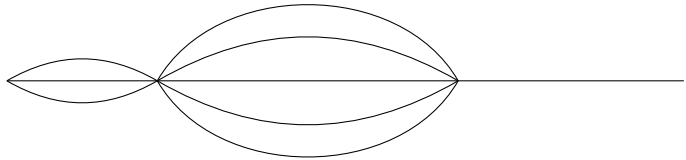
\begin{figure}
\label{fig:pumpkin-chain}
\begin{tikzpicture}
\coordinate (a) at (0,0);
\coordinate (b) at (2,0);
\coordinate (c) at (6,0);
\coordinate (d) at (9,0);
\draw (a) -- (b);
\draw[bend left]  (a) edge (b);
\draw[bend right]  (a) edge (b);
\draw (b) -- (c);
\draw[bend right]  (b) edge (c);
\draw[bend left=60]  (b) edge (c);
\draw[bend right=60]  (b) edge (c);
\draw[bend left]  (b) edge (c);
\draw (c) -- (d);
\end{tikzpicture}
\caption{A pumpkin chain on nine edges and four vertices.}
\end{figure}

In other words, $\mG$ consists of a chain of vertices, each one connected to a predecessor and a successor, all connecting two ``end vertices'', call them $\mv_0$ and $\mv_D$. Every path from $\mv_0$ to $\mv_D$ (of which there may be many) must pass through each of the $V-2$ interior vertices. Alternatively, a pumpkin chain may be thought of as a connected chain of pumpkin graphs; hence the name. See Figure~\ref{fig:pumpkin-chain}. Any two edges connecting the same pair of vertices (i.e.~belonging to the same pumpkin) will be called {\it parallel}.

The diameter of a pumpkin chain $\mG$ is given by $D(\mG) = {\dista}(\mv_0,\mv_D)$, and, if $\mG$ has at least three vertices, then the two end vertices are the only points realising $D$. 
Moreover, there is a canonical mapping $S: \mG \to [0,D]$ such that $S(\mv_0)=0$, $S(\mv_D)=D$, given by $S(z) = {\dista}(\mv_0, z)$. We shall call $S(z)$ the \emph{level} of a point 
$z \in \mG$.

The next result shows that for the purpose of estimating the spectral gap from above in terms of $V$ or $L$, we do not have to bother about any graphs other than pumpkin chains -- or, to put it in another way, no spectral information that is relevant to us is lost whenever we prune a quantum graph by deleting all edges that are not crossed by paths that realise the diameter.

\begin{lemma}
\label{lem:reduction}
Given any compact, connected, non-empty metric graph $\mG$, there exists a pumpkin chain $\tilde\mG$ such that
\begin{enumerate}
\item $D(\tilde\mG) = D(\mG)$, $L(\tilde\mG) \leq L(\mG)$ and $V(\tilde\mG) \leq V(\mG)+2$;
\item $\lambda_1(\tilde\mG) \geq \lambda_1(\mG)$.
\end{enumerate}
\end{lemma}

\begin{proof}
We will give an algorithm which constructs $\tilde\mG$ out of $\mG$ (note that we do not claim $\tilde\mG$ to be unique).

\noindent{\emph{Step 1.}} Choose any two points $x,y \in \mG$ such that ${\dista}(x,y) = D := D(\mG)$ and relabel them as $\mv_0, \mv_D$, i.e., relabel them as (artificial) vertices of $\mG$. Denote by $\Gamma_1$ any shortest path through $\mG$ connecting $\mv_0$ and $\mv_D$, that is, $\Gamma_1$ is a subgraph of $\mG$ with $L(\Gamma_1)=D(\mG)$, and any $\mv \in V(\mG) \cap \Gamma_1 \setminus \{\mv_0,\mv_D\}$ has degree two when considered as a vertex of $\Gamma_1$. Without loss of generality, $\Gamma_1$ may be considered as a directed path from $\mv_0$ to $\mv_D$, where no part of $\Gamma_1$, be it edge or vertex, is counted twice.

\noindent{\emph{Step 2.}} Assuming such a path to exist, denote by $\Gamma_2$ any (directed) shortest path from $\mv_0$ to $\mv_D$ such that $\Gamma_2 \not\subset \Gamma_1$ (as sets), and no point of $\mG$ is counted twice by $\Gamma_2$ (in particular $\Gamma_2$ does not contain any loops). Since $\mG$ is compact, if there is any such path, there will be one of minimal length. Repeat this process inductively to obtain $n$ such paths $\Gamma_1, \ldots, \Gamma_n \subset \mG$, that is, $\Gamma_k$ is a directed path from $\mv_0$ to $\mv_D$ without passing through any point twice, with
\begin{displaymath}
\mG_k \not\subset \bigcup_{j=1}^{k-1} \mG_j
\end{displaymath}
for all $k \geq 2$, and $D = L(\Gamma_1) \leq \ldots \leq L(\Gamma_k) \leq \ldots$. Since $\mG$ is compact, this process will terminate after $n \geq 1$ steps, that is, no further path $\Gamma \subset \mG$ can be found without either being contained in $\bigcup_{k\leq n} \Gamma_k$ or passing through at least one point of $\mG$ at least twice.

\noindent{\emph{Step 3.}} Set
\begin{displaymath}
	\hat{\mG} := \bigcup_{k=1}^n \Gamma_k,
\end{displaymath}
taking the vertex set of $\hat{\mG}$ to be 
\[
\mV(\hat{\mG}):=\{\mv_0,\mv_D\} \cup \{\mv\in\mV: \mv \text{ has degree $\geq 2$ in } \hat{\mG}\}\ .
\]
Then clearly $D(\hat{\mG})=D(\mG)$, $L(\hat{\mG})\leq L(\mG)$ and $V(\hat{\mG})\leq V(\mG)+2$. We also see directly that $\lambda_1(\hat{\mG}) \geq \lambda_1(\mG)$, as follows from Lemma~\ref{lem:principles}.(1): since any connected component of $\mG \setminus \hat{\mG}$ is connected to $\hat\mG$ at a single vertex (otherwise we could find another path not contained in $\Gamma_1,\ldots,\Gamma_n$ as in step two), adding this component to $\hat\mG$ can only decrease $\lambda_1$. Since $\mG \setminus \hat{\mG}$ can only consist of finitely many connected components, it follows that $\lambda_1(\hat{\mG}) \geq \lambda_1(\mG)$.

\noindent{\emph{Step 4.}} We shorten the paths $\Gamma_2,\ldots,\Gamma_n$ inductively to obtain new paths $\tilde\Gamma_2,\ldots,\tilde\Gamma_n$ so that (with $\tilde\Gamma_1=\Gamma_1$) $L(\tilde\Gamma_1) = L(\tilde\Gamma_2) = \ldots = L(\tilde\Gamma_3) = D$, without changing the topology of $\hat{\mG}$ (although some edges may contract to a point).

More precisely, we form a new graph $\tilde\mG_2$ from $\hat{\mG}$ by shortening $\Gamma_2 \setminus (\Gamma_1 \cup \mV)$ by length $L(\Gamma_2)-D \geq 0$; then $\lambda_1(\tilde\mG_2) \geq \lambda_1(\hat{\mG})$ by Lemma~\ref{lem:principles}.(4). Note that this may also shorten (a subset of) the other paths $\Gamma_3,\ldots,\Gamma_n$, but each will still have length at least $D$, and they will still be ordered by increasing length. We then proceed inductively, creating a graph $\tilde\mG_k$ out of $\tilde\mG_{k-1}$ by shortening $\Gamma_k \setminus \bigcup_{j\leq k-1} \tilde\Gamma_j$ to have length $D$. We end up with a graph $\tilde\mG_n$ consisting of $n$ possibly overlapping paths $\tilde\Gamma_1,\ldots,\tilde\Gamma_n$ from $\mv_1$ to $\mv_2$, such that $L(\tilde\Gamma_k) = D$ for all $k=1,\ldots,n$, with $D(\tilde\mG_n)=D$, $L(\tilde\mG_n) \leq L(\mG)$, $V(\tilde\mG_n) \leq V(\mG)+2$ and $\lambda_1(\tilde\mG_n) \geq \lambda_1(\mG)$.

\noindent{\emph{Step 5.}} Finally, we obtain a pumpkin chain $\tilde\mG$ out of $\tilde\mG_n$ by identifying all points at a given level where (at least) two of the paths have a common vertex. For each $k$, there exists a mapping $S_k: \tilde\Gamma_k \to [0,D]$ given by  $S_k(z)=d_{\tilde\Gamma_k}(z,\mv_0)$ for all $z \in \tilde\Gamma_k$, where $d_{\tilde\Gamma_k}$ denotes (Euclidean) distance measured along the path $\tilde\Gamma_k$. The vertices $\mv \in \mV(\tilde\mG_n) \cap \tilde\Gamma_k$ along the path $\tilde\Gamma_k$ have the form $\mv_0 = S_k^{-1}(0)$, $\mv_k^1=S^{-1}(x_k^1)$, \ldots, $\mv_k^{n_k} = S^{-1}(x_k^{n_k})$, $\mv_D = S_k^{-1}(D)$ for certain levels $x_k^1,\ldots,x_k^{n_k} \in (0,D)$, for each $k=1,\ldots,n$.

For every $x_k^j$, for all $j=1,\ldots,n_k$ and all $k=1,\ldots,n$, we insert an artificial vertex at each point $S_i^{-1}(x_k^j)$ for all $i=1,\ldots,n$, and then form $\tilde\mG$ out of $\tilde\mG_n$ by identifying for every fixed pair $j,k$ the vertices $S_i^{-1}(x_k^j)$ for all $i=1,\ldots,n$. The graph $\tilde\mG$ is then a pumpkin chain with the same diameter, length and number of vertices as $\tilde\mG_n$, in particular satisfying condition (1) of the lemma, and $\lambda_1(\tilde\mG) \geq \lambda_1(\tilde\mG_n)$ (cf.~Lemma~\ref{lem:principles}.(3)), so that (2) is satisfied as well.
\end{proof}

\begin{remark}
\label{rem:vertex-d}
It is clear that $ V(\tilde{\mG}) $ will be equal to $ V(\mG) $ if the diameter is realised as a distance between certain two vertices -- there will be no need to introduce artificial vertices in Step 1. In some cases it is more convenient to work with the combinatorial diameter $D_\mV$ (see \eqref{eq:comb-dv}), which we recall is the maximal distance within $\mG$ of any two \emph{vertices} of $\mG$: an inspection of the proof of Lemma~\ref{lem:reduction} reveals that it can be trivially modified to allow one to prove the existence, for any given $\mG$, of a corresponding pumpkin chain $\tilde\mG$ with
\begin{enumerate}
\item $D_\mV(\tilde\mG) = D_\mV(\mG)$, $L(\tilde\mG) \leq L(\mG)$ and $V(\tilde\mG) \leq V(\mG)$, and
\item $\lambda_1(\tilde\mG) \geq \lambda_1(\mG)$.
\end{enumerate}
All of our statements involving the diameter -- at least, all the negative statements as well as all the upper bounds in Section~\ref{sec:diam-pos} below -- also hold for $D_\mV$, as is easy to see. 

Note that if $ \mG $ is e.g.\ a pumpkin chain, a tree, a (higher dimensional) cube or a complete bipartite graph, see e.g.~\cite[Chapter~1]{Die05}, then $ D =D_\mV$, but this is not the case if $\mG$ is a complete graph.
\end{remark}

\subsection*{Reduction to a Sturm--Liouville problem.}

Next we prove that we can find an eigenfunction $\psi_1$ of $\lambda_1(\mG)$ (where $\mG$ is a pumpkin chain) which only depends on the level $S(z)$ of any $z \in \mG$, that is, $\psi_1$ takes on the same value on all parallel edges: this is the crucial argument that will allow us to study our problem by Sturm--Liouville arguments.

\begin{lemma}
\label{lemma:level}
Suppose $\mG$ is a pumpkin chain. There exists an eigenfunction $\psi_1$ associated with $\lambda_1(\mG)$ and a function $\varphi:[0,D]\to \R$ such that $\psi_1(z) = \varphi(S(z))$ for all $z \in \mG$.
\end{lemma}

\begin{proof}
{
If there exists an eigenfunction $\psi$ which does not vanish on at least one vertex of $\mG$, then we may construct a new eigenfunction $\psi_1$ by averaging the value of $\psi$ at each fixed level. It is easily checked that $\psi_1 \not\equiv 0$ created in this way is an eigenfunction for $\lambda_1(\mG)$, as a linear combination of eigenfunctions on each edge, which also satisfies the Kirchhoff condition at the vertices, and which by construction only depends on the level. So we merely need to ensure the existence of such an eigenfunction $\psi$.
}

{
Denote the vertices of $\mG$ by $\mv_0,\mv_1,\ldots,\mv_V=\mv_D$ ($V \geq 2$) and the $V-1$ pumpkin subgraphs of $\mG$ by $\mG_1,\ldots, \mG_{V-1}$, $\mG_k$ running from vertex $\mv_{k-1}$ to $\mv_k$. If there \emph{exists} an eigenfunction $\psi$ such that $\psi(\mv_k)=0$ for all $k=0,\ldots,V$, then, since $\psi$ cannot vanish identically on any pumpkin, $\psi$ is in particular an eigenfunction of the Dirichlet Laplacian on each $\mG_k$. Denoting by $\mu_n(\mG_k)$ the $n$th Dirichlet eigenvalue of $\mG_k$, it follows that for each $k$ there exists $n\geq 1$ with $\lambda_1(\mG) = \mu_n(\mG_k) \geq \mu_1(\mG_k) \geq \lambda_1(\mG)$, the last inequality following by an easy monotonicity argument. This implies $\mG$ must necessarily be equilateral, and $\psi$ restricted to each edge, identified with $(0,a) \subset \R$, is either identically $0$ or of the form $\pm\sin (\sqrt{\lambda_1} x) = \pm\sin (\pi x/a)$, $x \in (0,a)$. It follows that each $\mG_k$ has another eigenfunction taking on the value $\pm 1$ at $\mv_{k-1}$ and $\mp 1$ at $\mv_k$ (i.e., of the form $\pm \cos (\pi x/a)$ on each of its corresponding edges); we may therefore build another eigenfunction on $\mG$ taking on the alternate values $1$ and $-1$ at successive vertices. This proves the claim. (Note that it also follows that $\mV=2$ in this case, since for $\mV \geq 3$ one can see directly that any such eigenfunction cannot be associated with $\lambda_1$, but we do not need this.)
}
\end{proof}

 Let $ \psi_1 $ be such an eigenfunction as in the lemma, corresponding to $\lambda_1(\mG)$. We denote by 
\begin{equation} \label{rho}
\rho(x) = \# S^{-1}\{x\}
\end{equation}
 the counting function for the number of parallel edges of $\mG$ at the level $x \in [0,D]$, and by $\varphi(x):= \psi_1 (S^{-1}(x))$ 
 the well-defined composite function (in a slight abuse of notation) in $H^1(0,D)$; since $S'(z) = 1$ almost everywhere, it follows that
\begin{equation}\label{eq:1D}
\begin{split}
\lambda_1(\mG) &= \frac{\int_0^D \varphi'(x)^2\rho(x)\mathrm{d}x}{\int_0^D \varphi(x)^2\rho(x)\mathrm{d}x}\\
&=\inf \left\{ \frac{\int_0^D u'(x)^2\rho(x)\mathrm{d}x}{\int_0^D u(x)^2\rho(x)\mathrm{d}x}: u \in H^1(0,D),\,
	\int_0^D u(x)\rho(x)\mathrm{d}x =0 \right\};
\end{split}\end{equation}
for the second equality, ``$\geq$'' is clear, and ``$\leq$'' follows since every such function $u \in H^1(0,D)$ can be mapped canonically onto a function $g \in H^1(\mG)$ via $g(z) = u \circ S(z)$ for $z \in \mG$ (so that we may use the minimisation property of $\psi$ on $\mG$). We see in particular that looking for a pumpkin chain with the largest spectral gap it is enough to study a one-dimensional problem.

We shall now introduce a principle which allows us to pass from pumpkin chains to actual one-dimensional weighted eigenvalue problems. The central argument here is that the set of possible weights $\rho$ corresponding to pumpkin graphs, $\{\rho\in L^\infty(0,D;\mathbb N\setminus\{0\})\}$, is large enough within $L^\infty(0,D)$ in a sense relevant for our eigenvalue problems. This in turn is based on the scale invariance of \eqref{eq:1D}: nothing changes if we multiply $\rho$ by an arbitrary nonzero constant.

\begin{proposition}
\label{prop:weighted}
Given any function $\omega \in C^1[0,D]$ such that $\min_{x\in [0,D]} \omega(x)>0$, there exists a sequence of pumpkin chains $\mG_n$ of diameter $D$ such that $\lambda_1(\mG_n) \to \lambda_1(\omega)$.
\end{proposition}

Here we have denoted by $\lambda_1(\omega)>0$ the first non-trivial eigenvalue of the Sturm--Liouville problem
\begin{equation}\label{eq:weighted}
\left\{
\begin{aligned}
	-(\omega(x)u'(x))' &= \lambda \omega(x) u(x), \qquad x \in (0,D),\\
		u'(0) &= u'(D) = 0\ ,
\end{aligned}
\right.
\end{equation}
i.e., 
\begin{equation}
\label{eq:weighted-rq}
\lambda_1(\omega) = \inf \left\{ \frac{\int_0^D u'(x)^2\omega(x)\,\textrm{d}x}{\int_0^D u(x)^2\omega(x)\,\textrm{d}x} :
	u \in H^1(0,D),\,\int_0^D u(x)\omega(x)\textrm{d}x = 0 \right\},
\end{equation}
as is easy to see by the usual means. 

\begin{remark}
\label{rem:weighted}
We note that \eqref{eq:weighted-rq} makes sense when $\omega \in L^\infty(0,D)$ with $\essinf \omega(x) > 0$, and in this case it is still possible to find pumpkin chains $\mG_n$ with $\lambda_1(\mG_n) \to \lambda_1(\omega)$ -- as our proof shows, although we will not go into details. Hence, it follows from \eqref{eq:1D}, Lemma~\ref{lem:reduction} and Proposition~\ref{prop:weighted} that
\[
\begin{split}
	\sup\{\lambda_1(\mG): D(\mG)=D\} &= \sup\{\lambda_1(\mG): \mG \text{ pumpkin chain, } D(\mG)=D\} \\
	&= \sup\{\lambda_1(\omega):\omega\in L^\infty(0,D;\mathbb N\setminus{\{0\}})\}\ .
\end{split}
\]
We can in fact further refine the class of functions that are relevant to us: as noted, it follows directly from \eqref{eq:weighted-rq} that $\lambda_1(\alpha\omega) = \lambda_1(\omega)$ for any $\alpha>0$, which leads us finally to consider
\begin{displaymath}
W = \left\{ \rho \in L^\infty (0,D): \varepsilon_\rho :=\essinf \rho > 0,\right.\\ \left.\rho(0,D) \text{ is a finite subset of }
	\varepsilon_\rho\,\N \setminus \{0\} \right\},
\end{displaymath}
the set of all weights which can be obtained by multiplying the weight associated with a pumpkin chain by an arbitrary positive constant. In this way, our original spectral problem on a graph has been reduced to a suitable Sturm--Liouville-like problem with possibly discontinuous elliptic coefficients. We consider this transference principle to be the largest point of interest of this article, which can likely be applied to different investigations in the theory of quantum graphs.
\end{remark}

The proof of Proposition~\ref{prop:weighted} relies on the following lemma, whose proof is standard and hence omitted.

\begin{lemma}
\label{lem:weighted-conv}
Given $\omega_n,\omega \in L^\infty(0,D)$ with $\essinf \omega(x), \essinf \omega_n(x) \geq \varepsilon > 0$ for all $n \in \N$, if $\omega_n \to \omega$ in $L^\infty(0,D)$, then $\lambda_1(\omega_n) \to \lambda_1(\omega)$ (with $\lambda_1(\omega_n), \lambda_1(\omega)$ defined as in \eqref{eq:weighted-rq}).
\end{lemma}

In fact we can expect convergence in a much stronger sense, but this is all we will need.

\begin{proof}[Proof of Proposition~\ref{prop:weighted}]
Every finite step function $\rho \in L^\infty(0,D)$ taking on only positive integer values corresponds to a pumpkin chain; hence the proposition reduces to showing that we can find a sequence of functions $\rho_n\in W$ with $\lambda_1(\rho_n) \to \lambda_1 (\omega)$, with $\omega$ as in the statement of the proposition, since the closure of $W$ in the $L^\infty$-norm contains $\{ \omega \in C^1[0,D]: \min\omega > 0 \}$, as an elementary (and omitted) approximation argument shows. But this follows from Lemma~\ref{lem:weighted-conv}.
\end{proof}

\subsection*{Counterexample}

We are finally in a position to show that no upper estimate on the spectral gap is possible in terms of the diameter alone. Our result is non-constructive, but it should not be hard to write down a family of examples explicitly based on our proof: a sequence of pumpkin chains, where the pumpkins become increasingly small, and in each pumpkin chain the number of slices of the successive pumpkins increases exponentially from $\mv_0$ to $\mv_D$.

\begin{theorem}\label{th:d-upper}
Given $D>0$, there exists a sequence of pumpkin chains $\mG_n$ with $D(\mG_n) = D$ for all $n\in\N$, but $\lambda_1(\mG_n) \to \infty$.
\end{theorem}

\begin{proof}
Proposition~\ref{prop:weighted} together with a diagonal argument shows that this follows from the existence of a sequence of weights $\omega_n \in C^1[0,D]$ with $\lambda_1(\omega_n) \to \infty$. We make the explicit choice $\omega_n(x):= e^{nx}$, then it follows from \eqref{eq:weighted} and a short calculation that $\lambda_1(\omega_n)$ is the lowest non-zero eigenvalue of
\begin{equation}\label{eq:sturml}
\left\{
\begin{aligned}
	- u''(x) - n u'(x) &= \lambda u(x), \qquad x \in (0,D),\\
		u'(0) = u'(D) &= 0.
\end{aligned}
\right.
\end{equation}
We claim that all non-trivial eigenvalues $\lambda_k(\omega_n)$ of the above problem are bounded from below by $n^2/4$, from which the assertion of the theorem will follow. To see this we use the substitution $w(x):=e^{\frac{n}{2}x} u(x)$ to transform \eqref{eq:sturml} into
\begin{displaymath}
\left\{
\begin{aligned}
	- w''(x) & = \mu w(x) ,\qquad \mu =  \lambda - \frac{n^2}{4},  \qquad x \in (0,D),\\
	w'(0) &= \frac{n}{2} w(0), \,\, w'(D) = \frac{n}{2} w(D).
\end{aligned}
\right.
\end{displaymath}
An elementary calculation shows that every eigenvalue $\mu$ of this problem is at least $- \frac{n^2}{4} $; the unique eigenfunction corresponding to $ \mu_0 = - \frac{n^2}{4}   $ is $w(x)=e^{\frac{n}{2}x}$. 
The next eigenvalue $ \mu_1 $ is positive, meaning $\lambda_1(\omega_n) \geq n^2/4$.
\end{proof}

\begin{remark}
\label{rem:d-global}
The property $\lambda_1(\mG_n) \to \infty$ is a global property of the sequence of graphs, as can be seen as follows: create a new graph $\tilde\mG_n$ out of the graph $\mG_n$ from Theorem~\ref{th:d-upper} by attaching to its end vertex $\mv_D$ a single pendant edge of fixed length $\ell>0$. Then $D(\tilde\mG_n) = D+\ell$ and the spectral gap does not blow up any more; in fact $\lambda_1(\tilde\mG_n) \leq \pi^2/\ell^2$ for all $n\in\N$ by Lemma~\ref{lem:principles}.(1) and \eqref{lambda1path}.
\end{remark}

\subsection{Applications: the radius and the longest cycle of a graph} \label{sec:rad-cycle}

One could say that the principles set out in the preceding section are more general than simply applying to the quantity diameter. For example, in a natural analogy with the radius of a combinatorial graph, cf.~\cite[\S~1.3]{Die05}, we may define the \emph{radius} of a quantum graph $\mG$ by
\begin{displaymath}
	R(\mG ):= \inf_{x \in \mG } \, \sup_{y \in \mG} \, \dist(x,y),
\end{displaymath}
where as usual $\dist$ denotes the shortest (Euclidean) distance between points along paths within the graph, and as with diameter, infimum and supremum are taken over all points in the graph, not just vertices. This is fairly closely related to diameter, and indeed for the flower dumbbells in Example~\ref{ex:d-lower}, the combs in Remark~\ref{rem:combs} and all pumpkin chains, we have that $R(\mG)=D(\mG)/2$. In particular, radius alone leads to neither a lower nor an upper bound on $\lambda_1(\mG)$.

As another example, the \emph{length of a longest cycle} of a graph $\mG$ also yields neither a lower nor an upper bound on $\lambda_1(\mG)$ (unlike the length of the longest edge).
For the lower bound, if we take the sequence of graphs $\mG_n$ from Example~\ref{ex:d-lower} and modify each by inserting a second edge $\me_{00}$ of length $D/2$ between $\mv_1$ and $\mv_{-1}$, then we obtain a sequence of graphs having a longest cycle of length $D$ and spectral gap tending to zero, as we can see by extending $\psi_n$ by another sine curve in the obvious way on $\me_{00}$. For the upper bound, take a sequence of pumpkin chains $\mG_n$ as in Theorem~\ref{th:d-upper} and identify for each $n$ the two end vertices $\mv_0$ and $\mv_D$ to form a new graph $\mG_n'$ having a longest cycle of length $D$ and $\lambda_1(\mG_n') \geq \lambda_1(\mG_n)$ by Lemma~\ref{lem:principles}.(3).

\section{Estimates involving the diameter} \label{sec:diam-pos}

If beyond the diameter we place additional restrictions on our graphs, we can recover bounds. An interesting fact is that, unlike in the case of length (Example~\ref{ex:lev-lower}), bounding $V$ from above is now, in conjunction with $D$, enough to obtain an upper bound.

\begin{theorem}
\label{th:dv}
Let $\mG$ be a quantum graph having diameter $D>0$ and $V \geq 2$ vertices. Then
\begin{equation}
\label{eq:dv}
\lambda_1(\mG) \leq \frac{\pi^2}{D^2}(V+1)^2\ ,
\end{equation}
and even
\begin{equation}
\label{eq:dv-pu}
\lambda_1(\mG) \leq \frac{\pi^2}{D_\mV^2}(V-1)^2
\end{equation}
where $D_\mV$ is the combinatorial diameter of $\mG$ introduced in Remark~\ref{rem:vertex-d}.
\end{theorem}

\begin{proof}
In order to prove \eqref{eq:dv} it suffices to prove that
\begin{equation}
\label{eq:dv-pc}
	\sup \left\{\lambda_1(\mG): \mG \text{ is a pumpkin chain, } D(\mG) = D \text{ and } V(\mG)\leq V\right\} 
	\leq \frac{\pi^2}{D^2}(V-1)^2,
\end{equation}
since by Lemma~\ref{lem:reduction} for any graph $\mG$ we can find a pumpkin chain $\tilde\mG$ with $D(\tilde\mG)=D(\mG)$, $V(\tilde\mG)\leq V(\mG)+2$ and $\lambda_1(\mG)\leq\lambda_1(\tilde\mG)$; \eqref{eq:dv-pu} will follow from the same statement, appealing to Remark~\ref{rem:vertex-d}. We establish this via a somewhat crude test function argument. By the pigeonhole principle, at least two vertices $\mv,\mw$ of $\mG$, say at levels $a=S(\mv)<b=S(\mw) \in [0,D]$, must satisfy ${\dista}(\mv,\mw) = b-a \geq D/(V-1)$. Roughly speaking, we build a scaled cosine function on $(a,b)$ and extend this by $\pm 1$ to the rest of the graph. More precisely, we set
\begin{displaymath}
\psi(z):=
\begin{cases}
\alpha \qquad &\text{if $x=S(z) \in [0,a)$}\\
\alpha \cos(\pi (x-a)/(b-a)) \qquad &\text{if $x =S(z) \in [a,(a+b)/2]$}\\
\beta \cos(\pi (x-a)/(b-a)) &\text{if $x=S(z) \in [(a+b)/2,b]$}\\
-\beta &\text{if $x\in (b,D]$}
\end{cases}
\end{displaymath}
for all $z \in \mG$, where $x=S(z)$ throughout, and the constants $\alpha,\beta>0$ are chosen so that $\int_\mG \psi = 0$. It is easy to check that $\psi$ has Rayleigh quotient no larger than $\pi^2/(b-a)^2 \leq \pi^2(V-1)^2/D^2$, establishing \eqref{eq:dv-pc} and hence \eqref{eq:dv}.
\end{proof}

\begin{remark}
(a) A bound on the diameter $D$ is stronger than one on the total length $L$ (as $D \leq L$), but on the other hand a bound on the number $V$ of vertices is weaker than one on the number $E$ of edges, as an upper bound on $E$ implies one on $V$. Comparing Theorem~\ref{th:dv} with Theorem~\ref{th:le} seems therefore to be interesting. We recall (Example~\ref{ex:d-lower}) that $D$ and $V$ together are not enough for a lower bound.

(b) We do not expect the bounds in Theorem~\ref{th:dv} to reflect the optimal growth of $\lambda_1$ with respect to $V$; for example, we conjecture that $\left\{\lambda_1(\mG): D_\mV (\mG) = D_\mV \text{ and } V(\mG)\leq 4\right\}$ is equal to $4\pi^2/D_\mV^2$. We also expect that in general such optima can only be approximated and never attained, since (roughly speaking) one can always add more edges to such a graph without altering $D$ and $V$, but if added between the right vertices they should speed up diffusion on the graph.

(c) The bounds
\begin{equation}
\label{eq:disc-diam}
\frac{4}{VD_\mV \deg_{\max}}\le \alpha_1 \le \frac{1}{\deg_{\max{}}\deg_{\min{}}}\frac{4\log_2^2 V}{D_\mV^2}
\end{equation}
on the spectral gap $\alpha_1$ of the normalised Laplacian of a combinatorial graph in terms of combinatorial diameter $D_\mV$ and number of vertices are a direct consequence of those obtained in~\cite[Thm.~2.7]{AloMil85} and~\cite[Thm.~4.2]{Moh91b} for the unnormalised Laplacian, cf.\ Remark~\ref{rem:cannot}(c) for the definition of $\deg_{\max{}},\deg_{\min{}}$. As to be expected, these bounds are a positive function of $V$ and an inverse function of $D_\mV$. Interestingly, the lower bound in \eqref{eq:disc-diam} requires a combination of the maximal degree $\deg_{\max{}}$ together with $V$; the nature of the counterexamples in Example~\ref{ex:d-lower} suggests that the same may be true of the continuous case.
\end{remark}

\begin{remark}\label{rem:de}
Fixing the diameter $D$ and bounding the number of edges $E$ from above is enough to yield non-trivial upper and lower bounds on $\lambda_1(\mG)$. Indeed, since at least one edge must have length greater than or equal to $D/E$, we obtain the crude upper bound 
\[
\lambda_1(\mG)\leq \frac{4\pi^2 E^2}{D^2} ,
\]
while since $L \leq DE$, \eqref{eq:l-lower} implies the lower bound
\[
\lambda_1(\mG)\geq \frac{\pi^2}{D^2 E^2},
\]
with equality if $\mG$ is a loop in the first case and a path in the second. Both these bounds are however in general probably far from optimal. Since $\lambda_1(\mG)$ depends continuously on changes in the length of any given edge of $\mG$, even as the edge length tends to zero, it seems likely that there should exist maximisers and minimisers for each $D>0$ and $E\geq 1$ (as opposed to only a maximising sequence as in Theorem~\ref{th:dv}); we expect that our current bounds are only sharp if $E=1$. However, we leave this as an open problem, together with the much harder problem of determining the optimal constants and optimisers for each $E\geq 2$.
\end{remark}

\section{Estimates in terms of both the total length and the diameter} \label{sec:diam-length}

We will now prove bounds on the eigenvalue in terms of $D$ and $L$. We will start with the upper bound, which is the easier case.

\subsection{An upper bound in terms of $D$ and $L$.}

Interesting in the following upper bound is that the bound is a \emph{positive} function of $L$, a consequence of $L$ representing a finiteness condition on $\mG$ as a counterpart to $D$; it is entirely possible that there could be another upper bound on $\lambda_1$ which depends inversely on $L$ and positively on $D$.

Note that our upper estimate is again probably far from optimal, as we obtain it using another somewhat course test function argument, although it is ``sharp'' in the sense that we have equality when $L=D$, i.e., when $\mG$ is a path.

\begin{theorem}
\label{th:dl-upper}
Any quantum graph $\mG$ satisfying Assumption~\ref{ass:graph} and having diameter $D>0$ and total length $L \geq D$ satisfies
\begin{equation}
\label{eq:dl-upper}
	\lambda_1(\mG) \leq \frac{\pi^2}{D^2} \ \frac{4L-3D}{D}\ .
\end{equation}
\end{theorem}

\begin{proof}
Since the dependence of the right-hand side of \eqref{eq:dl-upper} on $L$ is positive, the claim will follow from Lemma~\ref{lem:reduction} if we can prove the corresponding statement for pumpkin chains. So assume that $\mG$ is indeed a pumpkin chain. Denoting as usual by $\mv_0$ and $\mv_D$ the terminal vertices of $\mG$, and by $S(z)=\dista(z,\mv_0)$ the level of the point $z \in \mG$, we construct a test function on $\mG$ by setting
\begin{displaymath}
	\psi(z) := 
\begin{cases}
A\cos\left(\frac{\pi S(z)}{D}\right) \qquad &\text{if $S(z)\leq D/2$}\\
B\cos\left(\frac{\pi S(z)}{D}\right) &\text{if $S(z)> D/2$}
\end{cases}
\end{displaymath}
for all $z \in \mG$, where $A$, $B>0$ are chosen to ensure $\int_\mG \psi = 0$. Denote by 
\[
\ell_\mathscr{L} := |\{z \in \mG: S(z)\leq D/2\}|- D/2
\]
and 
\[
\ell_\mathscr{R} := |\{z \in \mG: S(z)> D/2\}|- D/2
\]
the total extra length of $\mG$ not accounted for by the first path $\Gamma_1$ in $\mG$ realising the diameter in the left and the right half of $\mG$, respectively. Then $\ell_\mathscr{L} + \ell_\mathscr{R} = L-D$ by definition, and, estimating $\psi^2$ from below by zero and $(\psi')^2$ from above by $ A^2 $ and $ B^2 $ on the corresponding part(s) of $\mG$, we have
\begin{displaymath}
\begin{aligned}
\lambda_1(\mG) &\leq
\frac{\pi^2}{D^2}\ \frac{A^2\int_0^{D/2} \sin^2(\pi x/D)\,\textrm{d}x + A^2 \ell_\mathscr{L} + B^2 \int_{D/2}^{D} \sin^2(\pi x/D)\,\textrm{d}x + B^2 \ell_\mathscr{R}}{A^2\int_0^{D/2} \cos^2(\pi x/D)\,\textrm{d}x + B^2 \int_{D/2}^{D} \cos^2(\pi x/D)\,\textrm{d}x}\\
&= \frac{\pi^2}{D^2}\ \frac{(A^2+B^2)\frac{D}{4} + A^2 \ell_\mathscr{L} + B^2 \ell_\mathscr{R}}{(A^2+B^2)\frac{D}{4}}.
\end{aligned}
\end{displaymath}
Estimating both $\ell_\mathscr{L}$ and $\ell_\mathscr{R}$ from above by $L-D$ and rearranging yields \eqref{eq:dl-upper}.
\end{proof}

\subsection{A lower bound in terms of $D$ and $L$}

Finally, we will give a lower bound. It would appear that the minimum is given by the solution of a transcendental equation, which represents the first non-zero eigenvalue of a second-order problem on an interval with a non-standard boundary condition, in which the operator itself appears. This is often called a generalised Wentzell, or Wentzell--Robin, condition in the literature; see for example~\cite{AreMetPal03,MugRom07}.

\begin{theorem}
\label{th:dl-lower}
For any quantum graph $\mG$ satisfying Assumption~\ref{ass:graph} and having diameter $D>0$ and total length $L\geq 2D$, the first eigenvalue $\lambda_1(\mG)$ is at least as large as $\kappa^2$, where $\kappa>0$ is the smallest positive solution of the transcendental equation
\begin{equation}
\label{eq:dl-trans-2}
	\cos(2\kappa D) = (L-2D)\kappa\sin(2\kappa D).
\end{equation}
In particular,
\begin{equation}
\label{eq:dl-lower}
\lambda_1(\mG) \geq \frac{1}{2D(L-D)} > \frac{1}{2DL}.
\end{equation}
\end{theorem}

\begin{remark}\label{rem:lowerbd-dl}
(a) It is not clear if \eqref{eq:dl-trans-2} is optimal. It follows from Lemma~\ref{lem:ssd-computation} below that the optimal bound is not larger than the square of the first positive solution $\tilde\kappa>0$ of
\begin{equation}
\label{eq:dl-trans-1}
	\cos\left(\tilde\kappa \frac{ D}{2}\right) = \tilde\kappa \frac{L-D}{2} \sin\left(\tilde\kappa \frac{D}{2}\right),
\end{equation}
and it is not hard to show that
\begin{displaymath}
	\frac{1}{DL} < \frac{1}{DL-\frac{D^2}{2}} \leq \tilde\kappa^2 \leq \frac{3}{3DL-2D^2} < \frac{3}{DL}
\end{displaymath}
(for the upper bound, use the estimate $\cot x \leq 1/x - x/3$ for $x \in (0,\pi)$; for the lower bound, argue as in the proof of Theorem~\ref{th:dl-lower}, with $D$ in place of $2D$). In particular, the dependence in Theorem~\ref{th:dl-lower} on $D$ and $L$ is of the correct form.

(b) If $D \geq L/2$, then it follows from \eqref{eq:l-lower} that
\begin{displaymath}
	\lambda_1(\mG) \geq \frac{\pi^2}{L^2} \geq \frac{\pi^2}{2DL}.
\end{displaymath}
We see that lower bounds involving $L$ and $D$ are of most interest when $D$ is much smaller than $L$, and in this case Theorem~\ref{th:dl-lower} is applicable.

(c) Both equations \eqref{eq:dl-trans-2} and \eqref{eq:dl-trans-1} arise by reducing a graph to a path graph of length $D$ with ``mass'' equal to $L-D$ concentrated at the two vertices, and correspond to a one-dimensional Laplacian with generalised Wentzell boundary conditions at one endpoint. More precisely, consider the problem
\begin{equation*}
\begin{aligned}
	-u''(x) &= \kappa^2 u(x) \qquad \text{in $(0,D)$}\\
	u''(0)- \frac{2}{L-D}u'(0) &=0 \\
	u''(D)+ \frac{2}{L-D}u'(D) &=0 .
\end{aligned}
\end{equation*}
The first antisymmetric eigenfunction corresponds to the problem on the half interval  
\begin{equation*}
\begin{aligned}
	-u''(x) &= \kappa^2 u(x) \qquad \text{in $(0,D/2)$}\\
	u''(D/2)+ \frac{2}{L-D}u'(D/2) &=0 \\
	u(0)&=0.
\end{aligned}
\end{equation*}
Then any eigenfunction has the form $\psi(x) = \sin(\kappa x)$, and a short computation using the condition at $x=D$ shows that $\kappa$ satisfies \eqref{eq:dl-trans-1} if and only if it solves this problem. (For \eqref{eq:dl-trans-2} consider the same problem on $(0,2D)$.)
\end{remark}

Important for the proof will be the following class of graphs.

\begin{definition}\label{def:ssd}
A \emph{symmetric star dumbbell} graph (or SSD graph for short) $\mG$ of diameter $D$ and star size $\ell$ is a graph consisting of an edge $\me_0$ (the handle) of length $D-2\ell$ between two vertices $\mv_1$ and $\mv_2$, with identical star graphs $\mS_1$, $\mS_2$, each consisting of $m\geq 2$ edges each of length $\ell$ attached to $\mv_1$ and $\mv_2$, respectively. \end{definition}

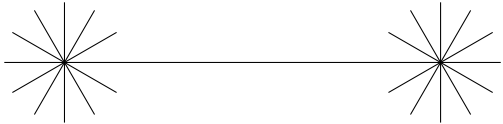
\begin{figure}
\begin{tikzpicture}
\begin{scope}
\coordinate (A) at (0,0);
\coordinate (D) at (5,0);
\draw (A) -- (D);
\draw (A) -- (30:0.8cm);
\draw (A) -- (60:0.8cm);
\draw (A) -- (90:0.8cm);
\draw (A) -- (120:0.8cm);
\draw (A) -- (150:0.8cm);
\draw (A) -- (180:0.8cm);
\draw (A) -- (210:0.8cm);
\draw (A) -- (240:0.8cm);
\draw (A) -- (270:0.8cm);
\draw (A) -- (300:0.8cm);
\draw (A) -- (330:0.8cm);
\end{scope}
\begin{scope}[xshift=5cm]
\draw (0,0) -- (30:0.8cm);
\draw (0,0) -- (60:0.8cm);
\draw (0,0) -- (90:0.8cm);
\draw (0,0) -- (120:0.8cm);
\draw (0,0) -- (150:0.8cm);
\draw (0,0) -- (210:0.8cm);
\draw (0,0) -- (240:0.8cm);
\draw (0,0) -- (270:0.8cm);
\draw (0,0) -- (300:0.8cm);
\draw (0,0) -- (330:0.8cm);
\draw (0,0) -- (360:0.8cm);
\end{scope}
 \end{tikzpicture}
\caption{A symmetric star dumbbell graph on twenty-two rays.}
\label{FigSSD}
\end{figure}

The next lemma shows that these graphs have a small first eigenvalue for given $D$ and $L$. Our graphs do not need to be \emph{star} dumbbells; it is merely important that as much mass as possible be concentrated as near the endpoints as possible. However, stars have the advantage that their eigenvalues can be more easily computed.

\begin{lemma}
\label{lem:ssd-comparison}
Given any quantum graph $\mG$ satisfying Assumption~\ref{ass:graph}, there is an SSD graph $\mT$ with $L(\mT)\leq L(\mG)$, $D(\mT)\leq \min\{2D(\mG),L(\mG)\}$, and $\lambda_1(\mT) \leq \lambda_1(\mG)$.
\end{lemma}

An inspection of our proof shows that if the graph $\mG$ is symmetric about a path representing the diameter, then one can find an SSD graph having the same diameter as $\mG$, with the other conclusions of Lemma~\ref{lem:ssd-comparison}. One can also find another graph $\mT'$ with $D(\mT') \leq D(\mG)$, $L(\mT')\leq 2L(\mG)$ and $\lambda_1(\mT') \leq \lambda_1(\mG)$. Moreover, if $\mG$ is already an SSD graph, then we can find an SSD graph $\mT''$ with $D(\mT'')=D(\mG)$, $L(\mT'')=L(\mG)$ and $\lambda_1(\mT'')\leq \lambda_1(\mG)$.

\begin{proof}
Given $\mG$ with length $L$ and diameter $D$, we take any eigenfunction $\psi$ corresponding to $\lambda_1(\mG)$. If $\psi$ vanishes identically on an edge, or a collection of edges, then we may remove them from $\mG$, decreasing $D$ and $L$, and leaving $\lambda_1$ unchanged. So we may assume that the set where $\psi = 0$ consists of finitely many points. We turn these into (artificial) vertices and identify them, which cannot increase $D$. By Lemma~\ref{lem:principles}.(2), this can only increase $\lambda_1$, but since $\psi$ remains an eigenfunction which is obviously still associated with the first non-trivial eigenvalue, as the new graph cannot have a smaller one, $\lambda_1$ remains in fact unchanged.

Hence we may assume that $\mG$ consists of two subgraphs $\mG^+$ and $\mG^-$, joined at a single vertex $\mv_0$, with a first eigenfunction $\psi$ satisfying $\psi>0$ on $\mG^+$ and $\psi<0$ on $\mG^-$. In particular, $\psi$ is also a first eigenfunction (with $\lambda_1$ the corresponding first eigenvalue) of the mixed Dirichlet--Neumann/Kirchhoff problem on $\mG^+$ and $\mG^-$, i.e., where we impose a Dirichlet condition at $\mv_0$ and the usual continuity and Kirchhoff conditions on all other vertices. Moreover, it is immediate that
\begin{displaymath}
	\sup\{\dist(x,\mv_0):x\in\mG^+\}, \, \sup\{\dist(x,\mv_0):x\in\mG^-\} \leq D\ ,
\end{displaymath}
and at least one of these graphs, say $\mG^+$, has total length not more than $L/2$. (Of course, one of them will also have this supremum at most $D/2$, but its length may then be large. Choosing this graph would lead to the comparison with $\mT'$.) Set 
\[
d:= \sup\{\dist(x,\mv_0):x\in\mG^+\} \leq \min\{D,L/2\}\ .
\]
Denote by $\mT$ an SSD graph having total length $L$ and diameter $2d$, where the length of the edges of the star graphs will be fixed later. Identify the midpoint (with respect to which $\mT$ is symmetric) with $0$, and the end vertices -- i.e., the vertices of degree one in $\mS_1$ and $\mS_2$, cf.\ Definition~\ref{def:hubgraph} -- with $-d$ and $d$, respectively. Denote the right half of the graph by $\mT^+$.

We will now construct a test function $\varphi$ on $\mT$ out of $\psi$. Let $z \in \mG^+$ be such that $\psi$ reaches its maximum $M>0$ in $\mG^+$ at $z$; then in particular $s:=\dist(z,\mv_0) \leq d$. Denote by $\Gamma$ any shortest path from $\mv_0$ to $z$. Now we map the values of $\psi$ on $\Gamma$ onto $\mT$: for $x \in [0,s] \subset [0,d]$ let $\varphi(x) = \psi(w)$, where $\dist(w,\mv_0)=x$, so that $\varphi(0)=0$ and $\varphi(s)=M$. If $s<d$, then extend $\varphi$ by $M$ on $(s,d]$, and on the star graph $\mS_2$ let $\varphi$ take on the value it takes on at $\mv_2$. We then take $\varphi$ to be symmetric on $\mT$; it is clear that $\varphi \in H^1(\mT)$. 

Now fix $\varepsilon>0$, to be specified later. We assume that $\mv_2$ is near enough to $d$ that $\varphi(\mv_2) \geq M-\varepsilon$. It remains to show that the Rayleigh quotient of $\varphi$ is not larger than that of $\psi$ on $\mG$; note that we only have to consider $\mG^+$, as it follows from the fact that $\psi$ is an eigenfunction and that $\mG^+$ is one of its nodal domains that
\begin{displaymath}
	\lambda_1(\mG) = \frac{\int_{\mG^+}|\psi'|^2}{\int_{\mG^+}|\psi|^2}.
\end{displaymath}
By symmetry, we also only have to consider $\varphi$ on $\mT^+$. Now it is clear from the construction that
\begin{displaymath}
	\int_{\mT^+} |\varphi'|^2 = \int_{(0,s)}|\varphi'|^2 = \int_\Gamma |\psi'|^2 \leq \int_{\mG^+}|\psi'|^2.
\end{displaymath}
We claim that
\begin{displaymath}
	\int_{\mT^+}|\varphi|^2 \geq \int_{\mG^+}|\psi|^2
\end{displaymath}
if $\varepsilon>0$ is small enough. If $d=L/2$, then $\mG^+$ and $\mT^+$ are both just path graphs, and it is clear there is equality between the corresponding eigenvalues. If $d<L/2$, then $L(\mG^+) \leq L(\mT^+) = L/2$; moreover, since $\psi$ reaches its maximum $M$ on a finite set of points, there exists $\delta>0$ such that
\begin{displaymath}
	\int_{\mG^+ \setminus \Gamma} |\psi|^2 \leq (M-\delta)^2 L(\mG^+ \setminus \Gamma)
	\leq (M-\delta)^2 (\frac{L}{2}-s).
\end{displaymath}
If we choose $\varepsilon \in (0,\delta)$, then, since $L(\mS_2 \cup (s,d)) = L/2 - s$,
\begin{displaymath}
	\int_{\mT^+\setminus (0,s)}|\varphi|^2 = \int_{\mS_2 \cup (s,d)} |\varphi|^2
	\geq (M-\varepsilon)^2 \left(\frac{L}{2}-s\right) > (M-\delta)^2 \left(\frac{L}{2}-s\right).
\end{displaymath}
Since clearly
\begin{displaymath}
	\int_{(0,s)}|\varphi|^2 = \int_\Gamma |\psi|^2,
\end{displaymath}
our claim follows, and we have shown that $\lambda_1(\mT) \leq \lambda_1(\mG)$. Since $\mT$ has total length $L$ and diameter $2d \leq 2D$, we are done.
\end{proof}

We now show that the spectral gap of a sequence $\mT_n$ of SSD graphs of fixed diameter and total length converges to the claimed equation as the stars become smaller and more concentrated.

\begin{lemma}
\label{lem:ssd-computation}
Suppose that $\mT_n$ is a sequence of SSD graphs with $D(\mT_n)=D$ and $L(\mT_n)=L$ for all $n$, such that $\ell_n \to 0$. Then $\lambda_1(\mT_n)$ is greater than and converges to $\kappa^2$, where $\kappa$ is the first positive solution of
\begin{displaymath}
	\cos \left(\kappa \frac{ D}{2}   \right)  = \kappa \frac{L-D}{2} \sin \left(\kappa \frac{ D}{2} \right).
\end{displaymath}
\end{lemma}

\begin{proof}
Assume that the SSD graph is formed by one edge of length $ d_n $ and $ 2 n $ ``short'' edges of lengths $ \ell_n $ attached to two vertices as shown in Figure \ref{FigSSD}. Elementary calculations allow one to
determine $ d_n $ and $ \ell_n $ through $ D $ and $ L $
$$ \left\{
\begin{array}{ccl}
D & = & 2 \ell_n + d_n \\
L & = & 2n \ell_n + d_n
\end{array} \right. \Rightarrow \left\{
\begin{array}{ccl}
\ell_n & = & \displaystyle \frac{L-D}{2 (n-1)} \\[3mm]
d_n & = & \displaystyle \frac{nD -L }{n-1}
\end{array} \right. . $$
We are interested in the case of large $ n $, when $ d_n \gg \ell_n .$
The lowest eigenvalue is $ \lambda_0 = 0 $ with the eigenfunction $ \psi_0 \equiv 1; $ for sufficiently large $ n $ the first nontrivial eigenfunction is antisymmetric with respect to the reflection in the middle point of the dumbbell, and hence zero at that point. By symmetry, it will also attain the same values on all smaller edges attached to each of the vertices, and the Kirchhoff condition implies its normal derivative is zero at every boundary vertex. Taking into account that the function is continuous and the sum of normal derivatives at the inner vertices is zero, we derive
 the following dispersion equation
$$ k \cot \left(k \frac{ nD-L}{2(n-1)}\right) = nk \tan \left(k \frac{ L-D}{2(n-1)}\right). $$
We are interested in its solution in the interval $ 0 < k < \displaystyle \frac{\pi}{D}. $ 
Multiplying the equation by $ \sin \left(k \frac{ nD-L}{2(n-1)}\right) \cdot \cos \left(k \frac{ L-D}{2(n-1)}\right) $ and using trigonometric identities, this equation can be written in an equivalent form
$$  g^n (k) := (1+n) \cos \left(k \frac{D}{2}\right) - (n-1) \cos \left(k \left( \frac{D}{2} - \frac{L-D}{n-1} \right)\right) = 0. $$
The function $ g^n (k) $ is obviously positive when $ k= 0 $ and $ k_1 $ is its first zero on the positive semi-axis. Using the fact that the graph of the cosine function on the interval $ (0, \pi/2) $ lies below its tangent line, we have
$$ \cos \left(k \left( \frac{D}{2} - \frac{L-D}{n-1} \right)\right) \leq \cos \left(k \frac{D}{2} \right) + k\frac{L-D}{n-1} \sin\left( k \frac{D}{2}\right)\cdot , $$
which in turn implies that
\begin{equation*}
\begin{split}
g^n(k) & \geq   (1+n) \cos \left(k \frac{D}{2}\right) - (n-1) \left( \cos \left(k \frac{D}{2}\right) + k  \frac{L-D}{n-1} \sin \left(k \frac{D}{2}\right) \right) \\
& = 2 \left( \cos \left(k \frac{D}{2}\right) - k \frac{L-D}{2} \sin \left(k \frac{D}{2}\right) \right). 
\end{split} 
\end{equation*}
It follows that $ k_1 $ lies to the right of the lowest positive solution to the equation
$$ g^\infty (k) :=  \cos \left(k \frac{D}{2}\right) - k \frac{L-D}{2} \sin \left(k \frac{D}{2}\right) = 0. $$
On the other hand $ k_1 $ approaches the first zero of $ g^\infty (k) $ as $ n \rightarrow \infty $, since $ g^n $ converges to $ g^\infty $ pointwise.
\end{proof}

\begin{proof}[Proof of Theorem~\ref{th:dl-lower}]
The proof of \eqref{eq:dl-trans-2} follows directly from combining Lemmata~\ref{lem:ssd-comparison} and \ref{lem:ssd-computation}. To see that \eqref{eq:dl-lower} follows from \eqref{eq:dl-trans-2}, we merely note that since $\cos x \geq 1 - x^2/2$ and $\sin x \leq x$ for all $x\geq 0$, it follows from \eqref{eq:dl-trans-2} that $2D(L-2D)\kappa^2 \geq 1-(2\kappa D)^2/2$; rearranging yields $\kappa^2 \geq 1/(2DL-2D^2)$.
\end{proof}

\section{Estimates on the spectral gap of the discrete Laplacian}\label{sec:discrete}

The \emph{normalised Laplacian} $\mathcal{L}_{\rm norm}$ of a graph was implicitly introduced in~\cite{Fie73}, but it has only enjoyed broad interest since Chung's thorough investigations of its properties, summarised in~\cite{Chu97}, which link combinatorics, spectral theory and geometry.
It is defined as the symmetric, $V\times V$-matrix whose diagonal entries are 1 and whose off-diagonal entry corresponding to the vertices $\mv,\mw\in \mV$ is given by $-(\deg(\mv)\deg(\mw))^{-\frac12}$, where $\deg(\mv)$ denotes the degree of vertex $\mv$, i.e., the number of edges incident to it. The normalised Laplacian always has real spectrum and in fact one can show by simple variational methods that its eigenvalues are all positive, cf.~\cite[Chapter 1]{Chu97}.
(Another popular matrix in graph theory is the \emph{discrete Laplacian}, thoroughly studied since~\cite{Fie73}, in which each non-zero off-diagonal entry of the normalised Laplacian is replaced by $-1$ and each diagonal entry equals minus the sum of the non-diagonal entries on the same row.)

In the introduction we observed that our estimates on $\lambda_1$ in the special case of equilateral metric graphs can be turned into estimates on the spectral gap $\alpha_1$ of the normalised Laplacian, and vice versa, by the formula~\eqref{eq:below}. Remarkably, some estimates obtained by studying quantum graphs turn into estimates on $\alpha_1$ that currently seem to be unavailable by purely combinatorial methods.

Let us write such estimates down for the case of the normalised Laplacian, which corresponds to equilateral quantum graphs with edges of unit length (in which case we have $L=E$). 
{
It follows from~\eqref{eq:below} that
\[
\alpha_1=1-\cos\sqrt{\lambda_1}\qquad \hbox{whenever }\lambda_1< \pi^2\ ,
\] 
hence $\alpha_1$ is a monotonically increasing function of $\lambda_1$ in the range $\lambda_1\in [0,\pi^2)$.}

For instance, from our first upper bound $0\le \lambda_1\le \frac{\pi^2 E^2}{L^2}=\pi^2$ we deduce the well-known fact that $\alpha_1$ always lies in $[0,2]$ -- in fact, so do \emph{all} eigenvalues of the normalised Laplacian. Furthermore, from our results in Sections~\ref{sec:diam-pos} and~\ref{sec:diam-length}, we can obtain the following upper bounds:
\medskip
\begin{center}
	\begin{tabular}{r|c|l}
\backslashbox{parameters}{$\alpha_1$} & available upper bound & holds if \\ \hline & &\\
$D_\mV$, $V$  & $\displaystyle 1-\cos \frac{\pi}{D_\mV}(V+1) $ & $V+1\le D_\mV$ \\[12pt] \hline 
& &\\
$D_\mV$, $E$ & $\displaystyle 1-\cos\left(\frac{\pi}{D_\mV}\sqrt{\frac{4E}{D_\mV}-3}\right)$ & $\displaystyle \sqrt{\frac{4E}{D_\mV}-3}\le D_\mV$
			\end{tabular}
	\label{tab:resume-2a-comb}
\end{center}
The first bound can only hold for paths, for which however the second bound is sharper (and in fact tight). Here we need the combinatorial diameter $D_\mV \leq D$ of Remark~\ref{rem:vertex-d}, since the diameter $D$ is meaningless in the case of combinatorial graphs. 

For equilateral quantum graphs with unit side length, which are the only ones we need to consider, we necessarily have $D \leq D_\mV + 1$, since points realising the diameter can be at distance no more than $1/2$ from a vertex. Hence we have the following estimates, which are valid for all graphs:
\begin{center}
	\begin{tabular}{r|c}
		\backslashbox{parameter}{$\alpha_1$} & available lower bound \\ \hline \\
$E$ & $\displaystyle 1-\cos\frac{\pi}{E}$ \\[12pt] \hline \\
$D_\mV$, $E$ & $\displaystyle 1-\cos\frac{\pi}{(D_\mV+1) E}$ \\[12pt] \hline \\
$D_\mV$, $E$ & $\displaystyle 1-\cos\frac{1}{\sqrt{(D_\mV+1) E}}$ \\[12pt]
			\end{tabular}
	\label{tab:resume-2b-comb}
\end{center}
The first bound is always better than the second (and in fact tight for paths, cf.~\cite[Example~1.4]{Chu97}), while the third -- which can be improved slightly by using \eqref{eq:dl-trans-2} rather than the simpler \eqref{eq:dl-lower} -- is better than the first if $E/(D_\mV+1) \geq \pi^2$ (e.g., for all complete graphs on more than six vertices).  

By estimating $\cos$ by a truncated power series expansion, we can summarise our bounds as follows:
\begin{equation}
\label{eq:ours}
\left.
\begin{array}{r}
\displaystyle \frac{\pi^2}{2E^2}-\frac{\pi^4}{24 E^4}<1-\cos\frac{\pi}{E}\\
\displaystyle \frac{1}{2(D_\mV+1)E} - \frac{1}{24(D_\mV + 1)^2 E^2}<1-\cos\frac{1}{\sqrt{(D_\mV+1) E}}\\
\end{array} \right\}
\le \alpha_1 \le 1-\cos\left(\frac{\pi}{D_\mV}\sqrt{\frac{4E}{D_\mV}-3}\right)< 
\frac{\pi^2}{2 D_\mV^2}\frac{4E-3D_\mV}{D_\mV} \ ,
\end{equation}
where the upper bound holds if $4E\le D^2_V+3D_V$ (this latter condition holds e.g.\ for all paths, for cycles on more than 6 vertices, but not for non-trivial complete graphs). Both the (non-truncated) lower bounds are asymptotically (in $E$) tight for cycles and hypercube graphs, cf.~\cite[Examples~1.5--1.6]{Chu97}. Moreover, in the particular case of path graphs, the first lower bound and the upper bound jointly give the spectral gap \emph{exactly}.

These can be compared with known estimates obtained by combinatorial means, such as
\begin{equation}
\label{eq:theirs}
\left.
\begin{array}{r}
\displaystyle \frac{1}{2D_\mV E}\\
\displaystyle \frac{1}{(\deg_{\max{}}+1)\deg_{\max{}}^{\lceil D_V/2 \rceil-1}}\\
\end{array} \right\}
\le \alpha_1\le
\left\{ 
\begin{array}{l}
\displaystyle \frac{V}{V-1}\\
\displaystyle  1-2\frac{\sqrt{\deg_{\max{}}-1}}{\deg_{\max{}}}\left(1-\frac{2}{D_\mV}\right)+\frac{2}{D_\mV}
 \end{array}
 \right.
\end{equation}
(see~\eqref{eq:spielman-revis},~\cite[Lemma~1.9 and Lemma 1.14]{Chu97} and \cite{ButChu13}), the first upper bound being tight for complete graphs and asymptotically (in $V$) tight for stars and complete bipartite graphs, cf.~\cite[Examples~1.1--1.3]{Chu97}, the second upper bound holding whenever $D_\mV\ge 4$. However, we see that the bounds in \eqref{eq:theirs} are not tight (in fact, two of them are not even asymptotically tight) on path graphs. Also, our upper bound yields
\[
1-\cos \frac{2\pi}{E}=\alpha_1\le 1-\cos \frac{2\pi\sqrt{5}}{E}
\]
for cycles of even length $E\ge 6$, to be compared with the combinatorial estimates by $\frac{E}{E-1}$ and $\frac{8}{E}$ in~\eqref{eq:theirs}. Hence, at least for some graphs, the bounds in \eqref{eq:ours} are actually sharper.

\section{Concluding remarks}\label{sec:concl}

It is easy to find perturbations of a graph which have an arbitrarily small effect on the first eigenvalue but which change the combinatorics of the graph enormously; we saw this phenomenon in action in Example~\ref{ex:lev-lower}. This is an easy but essential consequence of the principle that the eigenvalues depend continuously on changes in length of a given edge, including when that edge is contracted to a point, cf.\ the results in~\cite[\S~4]{DelRos14}. The moral is that, at least on a small or local scale, the analytic properties of a graph are more important for determining $\lambda_1$ than its combinatorial ones; in particular, quantum and combinatorial graphs can be expected to diverge considerably in their heuristic behaviour. This intuitive rule has been -- we believe -- underpinned by our results throughout this article. Nevertheless, we expect that other global structural properties of quantum graphs are also essential, even though our investigations have only touched on these aspects peripherally.

Let us summarise the best bounds we currently have:

\medskip

\begin{center}
	\begin{tabular}{r|c|l}
\backslashbox{parameter(s)}{$\lambda_1$} & available upper bound & is it sharp?  \\ \hline 
& &\\[-5pt]
$L$, $E$ & $\displaystyle \frac{\pi^2 E^2}{L^2}$ &  yes, attained on a pumpkin graph \eqref{lambda1pumpkin}\\[12pt] \hline 
& & \\
$D$, $V$  & $\displaystyle \frac{\pi^2}{D^2}(V+1)^2 $ & only if $V=1$, on a loop\\[12pt] \hline
& & \\
$D$, $E$  & $\displaystyle \frac{4\pi^2 E^2}{D^2}$ & only if $E=1$, on a loop\\[12pt] \hline
& & \\
$D$, $L$ & $\displaystyle \frac{\pi^2(4L-3D)}{D^3}$ & yes, attained on a path graph \eqref{lambda1path} \\ 
			\end{tabular}
	\label{tab:resume-2a}
\end{center}

\noindent
	and

\medskip

\begin{center}
	\begin{tabular}{r|c|l}
		\backslashbox{parameter(s)}{$\lambda_1$} & available lower bound & is it sharp?   \\ \hline
	& &\\[-2pt]
$L$ & $\displaystyle \frac{\pi^2}{L^2}$ &   yes, attained on a path graph \eqref{lambda1path}; \\
& & $V, E$ cannot improve it (Example~\eqref{ex:lev-lower})  \\[12pt] \hline 
& & \\
$D$, $E$  & $\displaystyle \frac{\pi^2}{D^2 E^2}$ & only if $E=1$, on a path graph \\[12pt] \hline
&& \\
$D$, $L$ & $\displaystyle \frac{1}{2DL}$ & unknown, but cf.\ Remark~\ref{rem:lowerbd-dl} \\[3pt] 
			\end{tabular}
	\label{tab:resume-2b}
\end{center}

\medskip

Of course, this should be viewed as being only a starting point: apart from the question of finding the optimal bounds and optimisers in many of the above cases, there are many other natural properties of a graph one could consider, as well as the higher eigenvalues, and the spectral problems which arise if one replaces the natural (Kirchhoff) boundary conditions with more general conditions on the vertices, such as more general $\delta$ or perhaps $\delta'$ couplings,  cf.~\cite{BerKuc13}, or impose Dirichlet condition on the vertices of degree one.

\bibliographystyle{alpha}

\end{document}